\documentclass[11pt]{amsart}
\usepackage{amsmath,amssymb,txfonts}
\usepackage{graphicx}
\usepackage{amscd,latexsym,amsthm,amsfonts,amssymb,amsmath,amsxtra}
\usepackage{amscd,mathrsfs,latexsym,amsmath,amsthm,amssymb,amsxtra,}
\usepackage{latexsym,amsmath,amsthm,amssymb,amsxtra,mathrsfs}
\usepackage[all]{xy}
\usepackage{amssymb}
\usepackage{amsxtra}
\usepackage{amsmath}
\usepackage{txfonts}
\newtheorem{prop}{Proposition}[section]
\newtheorem{thm}[prop]{Theorem}
\newtheorem{lem}[prop]{Lemma}
\newtheorem{claim}[prop]{Claim}

\newtheorem{defi}[prop]{Definition}
\begin{document}
\title{Cheeger-Colding-Tian theory for conic
K\"{a}hler-Einstein metrics}
\author{Gang Tian, Feng Wang}
\maketitle
\section{Introduction}

In a series of papers [CC1], [CC2], [CC3], Cheeger and Colding studied singular structures of spaces which arise as limits of sequences of Riemannian manifolds with Ricci curvature bounded below in the Gromov-Hausdorff topology. One of fundamental results they proved is the existence of tangent cones of the limit space [CC2], that is,

\begin{thm}([CC2])\label{thm-cc1}
Let $(M_i,g_i;p_i)$ be a sequence of $n$-dimentional Riemannian manifolds satisfying:
$${\rm Ric}_{M_i}(g_i)\,\geq\,{- (n-1)\Lambda^2g_i}~\text{and}~{\rm vol}_{g_i}( B_{p_i}(1))\,\geq\, v\,>\,0.$$
Assume that $(M_i,g_i;p_i)$ converge to a metric space $(Y,d;p_\infty)$ in the pointed Gromov-Hausdorff topology.  Then for any $y\in Y$ and sequence $\{r_j\}$ with $r_j\to 0$, there is a subsequence, say $\{\bar r_k= r_{j(k)}\}$, such that $(Y, \bar r_k^{-2} d; y)$ converge in the pointed Gromov-Hausdorff topology to a metric space $T_yY$ which is a metric cone over another metric space whose diameter is less than $\pi$. such a $T_yY$ is referred as a tangent cone of $Y$ at $y$.
\end{thm}

Note that the tangent cone $T_yY$ is not necessarily unique and may depend on the sequence $\{r_j\}$.
As an application of this theorem, Cheeger and Colding were able to introduce a stratification of singularities of the limit space $Y$.

\begin{defi}\label{singular-type}  Let $(Y,d;p_{\infty})$ be the limit of $(M_i,g_i;p_i)$ as in Theorem \ref{thm-cc1}. Denote by $\mathcal R$ the set of points which has a tangent cone isometric to $\mathbb R^n$ and $\mathcal S\,=\,Y\setminus \mathcal R$. For $k\leq n-1$, we say that $y\in\mathcal S_k$ if there exist no tangent cones at $y$ which can split off a Euclidean space $\mathbb R^l$ isometrically with $l>k$.
\end{defi}
Applying Theorem \ref{thm-cc1} to iterated tangent cones, Cheeger and Colding showed
\begin{thm}([CC2])
We have that $\mathcal S\,=\,\cup_{k=0}^{n-2}\, \mathcal S_k$ and $\dim \mathcal S_k\,\leq\, k$, where $\dim$ denotes the Hausdorff dimension.
\end{thm}
Based on the above theorem on existence of tangent cones, Cheeger, Colding and Tian [CCT] give further constraints on singularities of the limit space $Y$
under certain curvature condition for $(M_i,g_i)$ (also see Cheeger [Ch3]).

The purpose of this paper is to extend the Cheeger-Colding Theory to the following class of metrics.
This extension provides a technical tool for [LTW] in which we prove a version of the Yau-Tian-Donaldson conjecture for Fano varieties with certain singularity.

 \begin{defi}\label{manifold}
A length space $(M^n,d)$ is called a $n$-dimensional Riemannian manifold with singularity if there exists $S\subseteq M$ with $\mathcal H^n(S)\,=\,0$ such that the followings hold:\\

\noindent
i) $\mathcal R\,=\,M\setminus \mathcal S$ is a smooth manifold and convex, moreover, the distance function $d$ is induced from a smooth metric $g$ on $\mathcal R$.\\

\noindent
ii) for any $\epsilon>0$, denoting $T_\epsilon\,=\,\{x\,|\,dist(x, \mathcal S)\leq \epsilon\}$, there is a cut-off function $\gamma_\epsilon \in C^\infty_0(M\setminus \mathcal S)$ and
 $$\gamma_\epsilon\,\equiv\,1 \text{ in }~T_\epsilon(\mathcal S), ~\int_{\mathcal R}|\nabla \gamma_\epsilon|^2\,\leq \,\epsilon.$$\\

\noindent
iii) for any domain $U\,\subseteq\, M$ and a continuous function $b$ defined in a neighborhood of $\bar U$, there is a bounded function $h$ which is locally Lipschitz in $U$ and continuous in $\overline U\cap \mathcal R$ such that

 \[ \left \{ \begin{array}{l}
\Delta h \,=\, 0\text{ in }\mathcal R,\\
h|_{\partial U\bigcap \mathcal R}\,=\,b|_{\partial U\bigcap \mathcal R}.
\end{array} \right. \]
\end{defi}
We will study the limit space of the $n-$dimensional Riemannian manifolds with singularity whose Ricci curvature is bounded from below.
Let $\mathcal{M}(V,D,n)$ be the set of n-dimensional Riemannian manifolds $(M,d)$ with singularities satisfying:
$$\mathcal H^n(M)\,\geq\, V,~~ diam(M,d)\,\leq\, D,~~ {\rm Ric}(g)\,\geq\, 0\text{ in }\mathcal R.$$

Let $(M_i,d_i)$ be a sequence of manifolds in $\mathcal M(V,D,n)$ and $(M_i,d_i)\rightarrow (X,d)$. In this paper, we will prove
\begin{thm}\label{thm1}
For any $x\in X$ and sequence $\{r_j\}$ with $r_j\to 0$, there is a subsequence, say $\{\bar r_k= r_{j(k)}\}$, such that $(X, \bar r_k^{-2} d; x)$ converge in the pointed Gromov-Hausdorff topology to a metric space $T_xX$ which is a metric cone. Such $T_xX$ is referred as a tangent cone of $X$ at $x$. Moreover, there is a decomposition of $X$ into $\mathcal R\cup \mathcal S$ such that $\mathcal S=\mathcal S_{2n-2}$ and $\dim \mathcal S_k\leq k$, where $\mathcal S_k$ is defined as above.
\end{thm}
In \cite{Ba}, Bamler considered another class of singular spaces modeled on Ricci bounded space or Ricci flow. His definition of singular space is stronger.
Theorem \ref{thm1} could be also proved using the theory of RCD spaces developed by Ambrosio and others (\cite{A}, \cite{G}, \cite{Gi}, \cite{P}). Our proof here follows the approach of Cheeger-Colding by adapting their arguments to the conic case.

Let $M$ be a K\"ahler manifold and $D =\sum_{i=1}^k\,D_i$ be a normal crossing divisor. A metric $\omega$ is called a conic K\"ahler metric with conic angle $2\pi \beta_i$ along $D_i$, where $\beta_i\in (0,1)$, if it is a smooth K\"ahler metric outside $D$ and for each point $p\in D$ where $D$ is defined by the equation $z_1\cdots z_d = 0$ for some local coordinates $z_1,...,z_n$, $\omega$ satisfies
$$ C^{-1}\,\omega_{cone}\, \leq\, \omega \,\leq \,C\,\omega_{cone},$$
where $C$ is a positive constant and $\omega_{cone}$ is the model cone metric with cone angles $2\pi\beta_i$ along $\{z_i = 0\}$, that is,
$$\omega_{cone}\,=\,\sum_{i=1}^d \,\sqrt{-1}\,\frac{dz_i\wedge d \bar z_i}{|z_i|^{2(1-\beta_i)}}\,+\,\sum_{k=d+1}^n \,\sqrt{-1}\,dz_i\wedge d \bar z_i.$$

A conic K\"ahler metric is called a conic K\"{a}hler-Einstein metric on $M$ if for some constant $t$, $\omega$ satisfies:
$${\rm Ric}(\omega)\,=\,t\,\omega\,+\,2\pi\sum_{i=1}^k\,(1-\beta_i)\,[D_i],$$
where $[D_i]$ denotes the current defined by integrating $2n-2$-forms along $D_i$.

For any $\delta > 0$ and $V>0$, we denote by $
M(n,k,\delta,V)$ the set of all $n$-dimensional conic K\"{a}hler-Einstein metrics $(M,\omega)$ satisfying:
$$t\,\in\, [\delta, \delta^{-1}]~~~{\rm and}~~~\int_M\,\omega^n\,\geq\, V.$$
We will show
$$\mathcal M(n,k,\delta,V)\,\subseteq\, \mathcal M(V,\pi\sqrt{(2n-1)/\delta},2n),$$
consequently, we have the following:
\begin{thm}
For any limit space $X$ of conic K\"ahler-Einstein metrics in $\mathcal M(n,k,\delta,V)$, tangent cones of $X$ exist, that is,
for any $x\in X$ and sequence $\{r_j\}$ with $r_j\to 0$, there is a subsequence, say $\{\bar r_k= r_{j(k)}\}$, such that $(X, \bar r_k^{-2} d; x)$ converge in the pointed Gromov-Hausdorff topology to a metric space $T_x X$ which is a metric cone. Moreover, there is a decomposition of $X$ into $\mathcal R\cup \mathcal S$ such that $\mathcal S=\mathcal S_{2n-2}$, $\mathcal S_{2k+1}\,=\,\mathcal S_{2k}$ and $\dim \mathcal S_{2k}\,\leq\, 2k$.
\end{thm}

\section{Distance function comparison}
Let $(M,d)$ be an $n$-dimensional Riemmannian manifold with singularity which satisfies
$${\rm Ric}(g)\, \geq\, 0~~{\rm in}~~\mathcal R.$$
We will derive some basic estimates on $M$.
On $\mathcal R$, we have the Bochner formula:
\begin{align}\label{bochner-inequ}
\frac{1}{2}\Delta|\nabla f|^2\,=\,|{\rm Hess}\, f|^2\,+\,{\rm Ric}\,(\nabla f,\nabla f)\,+\,\langle \nabla f,\nabla \Delta f\rangle.
\end{align}
From this and the convexity of the regular part, the Laplacian comparison is the same as the smooth metric.
\begin{lem}\label{laplace}
For any $p \in \mathcal R$, $r(\cdot)=dist(p,\cdot)$ satisfies:
\begin{align}\label{mono-formula}
\Delta r\,\leq\,  \frac{n-1}{r}
\end{align}
in the sense of distribution in $\mathcal R$.
\end{lem}

As a consequence, we have
\begin{lem}
For any $p\in M$, the volume ratio $r^{-n}\,{\rm vol }(B_p(r))$ is monotone decreasing.
\end{lem}
\begin{proof}
First we assume $p\in \mathcal R$, then by the above lemma, in the sense of distribution,
we have
$$ \Delta r^2\,\leq \,2n~~~~{\rm on}~~\mathcal R.$$
Since the singular set $\mathcal S$ has zero volume, by the Fubini theorem, the $(n-1)$-dimensional Hausdorff measure of $\left (\partial B_p(s) \right )\cap \mathcal S $ vanishes for almost all $s\in [0,r]$. Then by using the convexity of $\mathcal R$ and arguing as in the smooth case, we can conclude
$$r^{-n}\,{\rm vol }(B_p(r))\,\le\, s^{-n}\,{\rm vol} (B_p(s))~~~~{\rm for ~any}~~s < r.$$
In general, when $s\le r$ is given, we choose a sequence of point $p_i\in \mathcal R$ converging to $p$, then we have
$$r^{-n}\,{\rm vol} (B_{p_i}(r))\,\le\, s^{-n}\,{\rm vol} (B_{p_i}(s)).$$
Taking the limit as $i$ goes to $\infty$, we get the required monotonicity.
\end{proof}
Using the convexity of the regular part, we can also show
\begin{lem}\label{equ-seg}
Let  $A_1, A_2$ be two  bounded subsets of $M$ and $W$ be another subset of $M$ satisfying:
$$\bigcup_{y_1\in A_1,y_2\in A_2}\gamma_{y_1y_2}\,\subseteq\, W,$$
where $\gamma_{y_1y_2}$ denotes a minimal geodesic connecting $y_1$ to $y_2$ in $M$. Put
$$D\,=\,sup\{\,d(y_1, y_2)~|~y_1\in A_1,y_2\in A_2\}.$$
Then for any  smooth  function $e$ on $W$, it holds
\begin{align}\label{segment-inequ}
&\int_{(A_1\cap \mathcal R)\times (A_2\cap \mathcal R)}\int_0^{d(y_1,y_2)}\,e(\gamma_{y_1,y_2}(s))\,ds\notag\\
&\leq c(n)\,D\,\left({\rm vol}(A_1)\,+\,{\rm Vol}(A_2)\right)\,\int_W \,e\,dv.
\end{align}
\end{lem}

\begin{proof} Note that
\begin{align}
\nonumber &\int_{(A_1\cap \mathcal R)\times (A_2\cap \mathcal R)}\,\int_0^{d(y_1,y_2)}\,e(\gamma_{y_1,y_2}(s))\,ds&\\
\nonumber &=\,\int_{A_1\cap \mathcal R}\,dy_1\,\int_{A_2\cap \mathcal R}\,\int_{\frac{d(y_1,y_2)}{2}}^{d(y_1,y_2)}\,e(\gamma_{y_1y_2}(s))\,dsdy_2\\
&\nonumber +\,\int_{A_2\cap \mathcal R}\,dy_2\,\int_{A_1\cap \mathcal R}\,\int_{\frac{d(y_1,y_2)}{2}}^{d(y_1,y_2)}\,e(\gamma_{y_1y_2}(s))\,dsdy_1.&
\end{align}
On the other hand, for a fixed $y_1\in A_1\cap \mathcal R$, by using  the monotonicity  formula  (\ref{mono-formula}), we have
\begin{align}
\nonumber &\int_{A_2\cap \mathcal R}\,\int_{\frac{d(y_1,y_2)}{2}}^{d(y_1,y_2)}\,e(\gamma_{y_1y_2}(s))\,dsdy_2\\
\nonumber&=\,\int_{A_2\cap \mathcal R}\,\int_{\frac{r}{2}}^{r}\,e(\gamma_{y_1y_2}(s))\,A(r,\theta)\,dr d\theta ds&\\
\nonumber &\leq\, c(n)\,\int_{A_2\cap \mathcal R}\,\int_{\frac{r}{2}}^{r}\,e(\gamma_{y_1y_2}(s))\,A(s,\theta)\,dr d\theta ds\\
\nonumber &\leq \,c(n)\,D\,\int_W\,e\,dv.&
\end{align}
Similarly,
\begin{align}
\nonumber &\int_{A_1\cap \mathcal R}\,\int_{\frac{d(y_1,y_2)}{2}}^{d(y_1,y_2)}\,e(\gamma_{y_1y_2}(s))\,dsdy_1\\
\nonumber &\leq \,c(n)\,D\,\int_W\,e\,dv.&
\end{align}
Then (\ref{segment-inequ}) follows from the above two inequalities.

\end{proof}
For any three points $x,y,z$, put
$$\int_{\gamma_{x,y}}\,e\,dv\,=\,\int_0^{d(x,y)}\,e(\gamma_{y_1,y_2}(s))\,ds$$
and
$$\int_{\Delta_{xyz}}\,=\,\int_{w\in\gamma_{xy}}\,\int_{\gamma_{zw}}\,e\,dv.$$
Then, by applying Lemma \ref{equ-seg} twice, we get
\begin{lem}\label{segment-inequ1}
Let $A_1, A_2, A_3$ be three  bounded subsets of $M$ and $W, Z$ be another two subsets of $M$ satisfying:
$$\bigcup_{x\in A_1,y\in A_2}\,\gamma_{y_1y_2}\,\subseteq \,W~~~~~{\rm and}~~~~~\bigcup_{w\in W,z\in A_3}\,\gamma_{zw}\,\subseteq\, Z .$$
Then for any smooth function $e$ on $Z$, it holds
\begin{align}\label{segment-inequ2}
&\int_{(A_1\cap \mathcal R)\times (A_2\cap \mathcal R)\times (A_3\cap \mathcal R)}\,\int_{\Delta_{xyz}}\,e\,dv\\
&\leq \,c(n)\,diam(W)\,diam(Z)\,\left({\rm vol}(A_1)\,+\,{\rm vol}(A_2)\right)\,\left({\rm vol}(A_3)\,+\,{\rm vol}(W)\right)\,\int_Z\,e\,dv.\notag
\end{align}
\end{lem}
\begin{lem}\label{max-principle}
Let $u$ be a bounded function in a bounded domain $\Omega$. Assume that $u$ is harmonic in $\Omega\bigcap \mathcal{R}$ and $u\leq 0$ on $\partial\Omega$. Then $u\leq 0$ in $\Omega$.
\end{lem}
\begin{proof}
At first, we deal with the special case when $u=0$ on $\partial\Omega$. Then we have
\begin{align}
\notag \int_{\Omega\cap \mathcal{R}}\,|\nabla u|^2\,\gamma_\epsilon^2
&=\,- 2\,\int_{\Omega\cap \mathcal{R}} \,u\,\gamma_\epsilon\,\langle \nabla u,\nabla \gamma_\epsilon \rangle \,-\,\int_{\Omega\cap \mathcal{R}}\,u \,\gamma_\epsilon^2\,\Delta u\\
\notag &\leq\, \frac{1}{4}\,\int_{\Omega\cap \mathcal{R}}\,|\nabla u|^2\,\gamma_\epsilon^2\,+\,4\,
\int_{\Omega\cap \mathcal{R}}\,|u|^2\,|\nabla \gamma_\epsilon|^2.
\end{align}
So we have
$$\int_{\Omega\cap \mathcal{R}}\,|\nabla u|^2\,\gamma_\epsilon^2\,\leq\, C\,\int_{\Omega\cap \mathcal{R}}\,|\nabla \gamma_\epsilon|^2.$$
Taking $\epsilon \rightarrow 0$, we get
$$\int_{\Omega\cap \mathcal{R}}\,|\nabla u|^2\,=\,0$$
which implies that $u\equiv 0$ in $\Omega$.

Now we consider the general case. If there is a point $p\in \Omega\bigcap \mathcal{R}$ such that $u(p)>0$, then $\Omega'=\{x\,|\,u(x)>0\}$ is a non-empty domain.
Since $u$ vanishes on the boundary of $\Omega'$, we deduce from the above special case that $u\equiv 0$ on $\Omega'$. It is a contradiction.
The lemma is proved.
\end{proof}
Note that by using the cut-off function $\gamma_\epsilon$ in ii), we can show that integration by parts holds on $M$.
\begin{lem}\label{part1}
Assume that $\Omega$ is a bounded domain, $\phi\in C^0(\Omega)\cap C^\infty(\Omega\cap \mathcal R)$ and $u\in L^\infty(\Omega)\cap C^\infty(\Omega\cap \mathcal R)$ satisfying:
$$|u|_{L^\infty(\Omega)}+|\nabla \phi|_{C^0(\Omega\cap \mathcal R)}+|\Delta \phi|_{C^0(\Omega\cap \mathcal R)} \leq C.$$  If
$$\int_{\Omega\cap \mathcal R} \phi^2|\nabla u|^2<\infty,$$ we have

$$\lim_{\epsilon\rightarrow 0}\int_\Omega \phi\gamma^2_\epsilon \Delta u=\int_{\Omega\cap \mathcal R} u\Delta\phi.$$
\end{lem}
\begin{proof}
Using integration by parts, we get
\begin{align}\label{intpart}
\int_\Omega \phi\gamma^2_\epsilon \Delta u=\int_\Omega u \gamma_\epsilon^2\Delta\phi +2\int_\Omega\langle \nabla \phi,\nabla \gamma_\epsilon\rangle u\gamma_\epsilon-2\int_\Omega\langle \nabla u,\nabla \gamma_\epsilon\rangle \phi\gamma_\epsilon
\end{align}
Since
$$|\int_\Omega\langle \nabla \phi,\nabla \gamma_\epsilon\rangle u\gamma_\epsilon|\leq C\int_\Omega|\nabla \gamma_\epsilon|\rightarrow 0$$ and
$$|\int_\Omega\langle \nabla u,\nabla \gamma_\epsilon\rangle \phi\gamma_\epsilon|\leq \left(\int_\Omega|\nabla\gamma_\epsilon|^2\int_\Omega\phi^2|\nabla u|^2\right)^\frac{1}{2},$$
we get the result.
\end{proof}
The integration condition can be obtained by applying the Bochner formula.
\begin{lem}\label{part2}
Assume that $\Omega$ is a bounded domain, $\phi\in C^0(\Omega)\cap C^\infty(\Omega\cap \mathcal R)$ and $ u\in L^\infty(\Omega)\cap C^\infty(\Omega\cap \mathcal R)$ satisfying:
$$|u|_{L^\infty(\Omega)}+|\nabla \phi|_{C^0(\Omega\cap \mathcal R)}+|\Delta \phi|_{C^0(\Omega\cap \mathcal R)} \,\leq\, C.$$
If $\Delta u\,\geq\, c\,|\nabla u|^2$ in $\Omega\cap\mathcal R$ for some $c> 0$, we have
$$\int_{\Omega\cap \mathcal R} \phi|\nabla u|^2\,<\,\infty.$$
\end{lem}
\begin{proof}
From (\ref{intpart}), we have
\begin{align}\label{q}
\int_\Omega \phi\gamma^2_\epsilon \Delta u\,\leq\, \int_\Omega u \gamma_\epsilon^2\Delta\phi+C\int_\Omega|\nabla \gamma_\epsilon|+2\left(\int_\Omega|\nabla\gamma_\epsilon|^2\int_\Omega\phi\gamma_\epsilon^2|\nabla u|^2\right)^\frac{1}{2}.
\end{align}
Since $\Delta u\,\geq\, c\,|\nabla u|^2$, we have
$$c\,\int_\Omega \phi\gamma^2_\epsilon |\nabla u|^2 \,\leq\, \int_\Omega \phi\gamma^2_\epsilon \Delta u.$$
Then the required estimate follows from (\ref{q}) and the Cauchy-Schwarz inequality.
\end{proof}

Now we use the Moser iteration to derive the gradient estimate for harmonic functions. See \cite{HKX} for the gradient estimate of harmonic functions on RCD spaces.
\begin{lem}\label{gradient-esti}
Let $u>0$ be a harmonic function defined on the unit ball $B_p(1)$, i.e.,
 $$\Delta u\,=\,0, ~ \text{in }~B_p(1)\cap \mathcal R.$$
Then
 \begin{align}
 |\nabla u|^2\,\leq\, C(n)\,u^2, ~\text{in}~ B_p(1/4)\cap \mathcal R.
 \end{align}
 \end{lem}
\begin{proof}
Putting $v=\ln u$, we have
$$\Delta v=\frac{\Delta u}{u}-\frac{|\nabla u|^2}{u^2}\,=\,-|\nabla v|^2.$$
Denote $Q\,=\,|\nabla v|^2$, by the Bochner formula, we have
\begin{align}\label{max}
\frac{1}{2}\Delta Q&=\,|{\rm Hess} \,v|^2\,-\,\langle \nabla v,\nabla Q\rangle\,+\,{\rm Ric}(\nabla u,\nabla u)\,\geq\, \frac{(\Delta v)^2}{n}\,-\,Q^{\frac{1}{2}}|\nabla Q|\\\notag
&=\,\frac{Q^2}{n}\,-\,Q^{\frac{1}{2}}|\nabla Q|.
\end{align}
For any Lipschitz function $\phi$ supported in $B_p(1)$, we have
$$\int_{B_p(1)}\,\phi^2\gamma_\epsilon^2Q^{p-1}\Delta Q\,\geq\, \frac{2}{n}\int_{B_p(1)}\, \phi^2\gamma_\epsilon^2 Q^{p+1}\,-\,2\int_{B_p(1)}\phi^2\gamma_\epsilon^2Q^{p-\frac{1}{2}}|\nabla Q|.$$
Integrating by parts, we have
\begin{align}
&2\int_{B_p(1)}\,\phi^2\gamma_\epsilon^2\,\left(Q^{p-\frac{1}{2}}|\nabla Q|\,-\,\frac{1}{n}\, Q^{p+1}\right)\\
\geq &\int_{B_p(1)}\,\left((p-1) \phi^2\gamma_\epsilon^2Q^{p-2}|\nabla Q|^2\,+\,\gamma_\epsilon^2Q^{p-1}\langle \nabla\phi^2,\nabla Q\rangle
\,+\,\phi^2Q^{p-1}\langle \nabla \gamma_\epsilon^2,\nabla Q\rangle\right) \notag
\end{align}
Since

$$\int_{B_p(1)}\phi^2Q^{p-1}|\langle \nabla \gamma_\epsilon^2,\nabla Q\rangle|
\leq \delta\int_{B_p(1)}\phi^2\gamma_\epsilon^2Q^{p-2}|\nabla Q|^2+
\delta^{-1}\int_{B_p(1)}\phi^2|\nabla \gamma_\epsilon|^2Q^{p-2}, (\forall \delta>0)$$ and $Q,|\nabla \phi|$ are bounded, taking $\epsilon\rightarrow 0$ and then $\delta\rightarrow 0$, we get
\begin{align}
\frac{4(p-1)}{p^2}\int_{B_p(1)\bigcap \mathcal R}\phi^2|\nabla Q^{\frac{p}{2}}|^2&\leq \frac{4}{p}\int_{B_p(1)\bigcap \mathcal R}\phi|\nabla \phi||\nabla Q^{\frac{p}{2}}|Q^{\frac{p}{2}}+\frac{4}{p}\int_{B_p(1)\bigcap \mathcal R}\phi^2|\nabla Q^{\frac{p}{2}}|Q^{\frac{p+1}{2}} \notag\\
&-\frac{2}{n}\int_{B_p(1)\bigcap \mathcal R}\phi^2Q^{p+1}. \notag
\end{align}
Consequently, we obtain
$$\int_{B_p(1)\bigcap \mathcal R}\phi^2|\nabla Q^{\frac{p}{2}}|^2\leq 9\int_{B_p(1)\bigcap \mathcal R}|\nabla \phi|^2 Q^p+9\int_{B_p(1)\bigcap \mathcal R}\phi^2 Q^{p+1}-\frac{p}{2n}\int_{B_p(1)\bigcap \mathcal R}\phi^2 Q^{p+1}.$$
Then we have
\begin{align}\label{inequality}
\int_{B_p(1)\bigcap \mathcal R}|\nabla (\phi Q^{\frac{p}{2}})|^2\,\leq\, \int_{B_p(1)\bigcap \mathcal R}\left(20\, |\nabla \phi|^2 Q^p\,+\,20\,
\phi^2 Q^{p+1}\,-\,\frac{p}{n}\,\phi^2 Q^{p+1}\right)
\end{align}
So for $p_1=40n$, we have
\begin{align}\label{es1}
\int_{B_p(1)\bigcap \mathcal R}|\nabla (\phi Q^{\frac{p_1}{2}})|^2\,\leq\, 20 \int_{B_p(1)\bigcap \mathcal R}|\nabla \phi|^2 Q^{p_1}-20\int_{B_p(1)\bigcap \mathcal R}\phi^2 Q^{p_1+1}
\end{align}
Let $\psi$ be a cut-off function supported in $B_p(\frac{1}{2})$ satisfying
$\psi\equiv 1$ in $B_p(1)$ and $|\nabla \psi|\leq 4$.
Put $\phi=\psi^{p_1+1},$ we have
$$|\nabla \phi|^2\leq 16(p_1+1)^2 \phi^{\frac{2p_1}{p_1+1}}.$$
Combined with the H\"{o}lder inequality, we get
\begin{align}\label{es2}
 \int_{B_p(1)\bigcap \mathcal R}|\nabla \phi|^2 Q^{p_1}
 &\leq \, C(n)\,\int_{B_p(1)\bigcap \mathcal R}\phi^{\frac{2p_1}{p_1+1}}Q^{p_1}\notag\\
 &\leq \,C(n)\,\left (\int_{B_p(1)\bigcap \mathcal R}\phi^2 Q^{p_1+1}\right )^{\frac{p_1}{p_1+1}}(vol(B_p(1))^{\frac{1}{p_1+1}} \notag\\
 &\leq\, \frac{1}{2}\,\int_{B_p(1)\bigcap \mathcal R}\phi^2 Q^{p_1+1}\,+\,C(n)vol(B_p(1)).
 \end{align}
 By the H\"{o}lder inequality, we have
 \begin{align}\label{es3}
 \int_{B_p(1)\bigcap \mathcal R}\phi^2Q^{p_1}&\leq \,\left(\int_{B_p(1)\bigcap \mathcal R}\phi^2Q^{p_1+1}\right)^{\frac{p_1}{p_1+1}}\left(\int_{B_p(1)\bigcap \mathcal R}\phi^2\right)^{\frac{1}{p_1+1}}\notag\\
 &\leq\, \frac{1}{2}\,\int_{B_p(1)\bigcap \mathcal R}\phi^2 Q^{p_1+1}\,+\,C(n)vol(B_p(1)).
 \end{align}
Combined with (\ref{es1}), (\ref{es2}) and (\ref{es3}), we can apply the Sobolev inequality to obtain
 \begin{align}\label{finite}
 \left(\frac{\int \phi^{2\gamma}Q^{p_1\gamma}}{vol(B_p(1))}\right)^\frac{1}{\gamma}\,\leq\, \frac{C(n)}{vol(B_p(1))}\,\int_{B_p(1)\bigcap \mathcal R}\left (|\nabla (\phi Q^{\frac{p_1}{2}})|^2\,+\,\phi^2Q^{p_1}\right)\,\leq\, C(n).
 \end{align}

For $p\geq 20$, we deduce from (\ref{inequality})
\begin{align}\label{iteration}
\int_{B_p(1)\bigcap \mathcal R}|\nabla (\phi Q^{\frac{p}{2}})|^2\leq 20\int_{B_p(1)\bigcap \mathcal R}|\nabla \phi|^2 Q^p
\end{align}

Using (\ref{finite}), (\ref{iteration}) and Moser's iteration, we get
$$|Q|_{L^\infty(B_p(\frac{1}{4}))}\,\leq\, C.$$

\end{proof}

\begin{lem}\label{cut-off}
For $p\in \mathcal R$, there exists a cut-off  function $\phi$ supported in $B_p(2)$ such that
i)  $\phi\equiv1$,  in  $B_p(1)$;  ii)
\begin{align}
 |\nabla \phi|_{B_p(2)\cap \mathcal R}, \,|\Delta\phi|_{B_p(2)\cap \mathcal R} \,\le \, C(n).\notag
\end{align}
\end{lem}
\begin{proof} We will use an argument from Theorem  6.33 in [CC1].
First we consider a solution  of ODE,
\begin{align}
G''+\frac{2n-1}{r}G'=1, ~\text{on}~ [1,2],
\end{align}
with $G(1)=a$  and  $G(2)=0$.   When $a\geq a(n)$, we have $G'<0$.
Then  by Lemma (\ref{laplace}),  we have
$$\Delta G(d(p,\cdot))\geq1.$$
Let $w$ be a  solution of equation,
\begin{align}
\Delta w=\frac{1}{a}, ~\text{in}~B_p(2)\setminus \overline{B_p(1)},\notag \end{align}
with $w=1$  on  $\partial B_p(1)$ and  $ w=0$  on  $\partial B_p(2)$.
 Thus  by Lemma \ref{max-principle},  we get
  $$w\geq \frac{G(d(.,p))}{a}.$$

Secondly,  denote $H=\frac{r^2}{4n}$. Then  by (\ref{laplace}),   we have
$$\Delta H(d(x,\cdot))\leq 1,\text{ for any fixed point} ~ x. $$
Thus by  the maximum principle,  we get
$$w(y)-\frac{H(d(x,y))}{a}\leq max\{1-\frac{H(d(x,p)-1)}{a},0\}$$
for any $y$ in the annulus  $A_p(1,2)=B_p(2)\setminus \overline{B_p(1)}$.
It follows
$$w(x)\leq max\{1-\frac{H(d(x,p)-1)}{a}, 0\},~\forall~ x\in A_p(1,2).$$

 Now  we choose a number  $\eta(n)$ such that $\frac{G(1+\eta)}{a}>1-\frac{H(1-\eta)}{a}$ and we  define a function $\psi(x)$ on $[0,1]$  with bounded derivative up to second order,  which satisfies
\begin{align}
\psi(x)=1,  \text{ if } x\geq \frac{G(1+\eta)}{a}\notag
\end{align}
 and
 \begin{align}
  \psi(x)=0, \text{ if } x\leq max\{1-\frac{H(1-\eta)}{a}, 0\}.\notag
\end{align}
 It is clear that  $\phi=\psi\circ w$ is constant near the boundary of $A_p(1,2)$.  So we  can extend  $\phi$  inside $B_p(1)$ by setting $\phi=1$.  By Proposition  \ref{gradient-esti},  one sees that  $|\nabla\phi|$ is bounded by a constant $C(n,\Lambda, A)$
 in $B_2(p)$.    Since
 $$\Delta \phi=\psi''|\nabla w|^2+\psi'\Delta w, $$
 we also derive  that $|\Delta \phi| \le C(n)$.
\end{proof}
\section{splitting theorem}
Let $(M_i,p_i)\in \mathcal{M}(V,D,n)$ be a sequence of Riemannian manifold with singularity and converge to $(X,x)$ in the pointed Gromov-Hausdorff sense. In this section, we will prove
\begin{prop}\label{split}
If $X$ contains a line, then there exists a length space $Y$ such that
$$X\cong Y\times\mathbb{R}.$$
\end{prop}
As in [CC1], the proof depends on the following lemmas.
\begin{lem}\label{excess}
Let $M$ be a Riemannian manifold with singularity with ${\rm Ric}(g)\,\geq\, 0$ in $\mathcal R$.
  Suppose that there are three points $p, q^+, q^-\in \mathcal R$ which satisfy
\begin{align}
d(p,q^+)+d(p,q^-)-d(q^+,q^-)<\epsilon\end{align}
and
\begin{align}d(p,q^+),d(p,q^-)> R.
\end{align}
Then  for  any $q \in B_p(1)$,  the following  holds,
$$E(q):=d(q,q^+)+d(q,q^-)-d(q^+,q^-) <\Psi(\epsilon,\frac{1}{R};n), $$
where the quantity  $\Psi(\epsilon, \frac{1}{R}; n)$ means that it goes to zero as   $\epsilon, \frac{1}{R}$  go to zero  while $n$ is  fixed.
\end{lem}
\begin{proof}
By Lemma (\ref{laplace}), we have $\Delta E(q)\leq \frac{4n-2}{R}$. Put
\begin{align} \label{aux}
G_L(r)=\frac{r^2}{4n}+\frac{L^{2n}}{4n(n-1)}r^{2-2n}-\frac{L^2}{4(n-1)}.
 \end{align}
 $G_L$ satisfies
$$G_L'<0, G_L(L)=0, \Delta G_L(d(p,\cdot))\geq 1.$$
We will prove
 \begin{claim}\label{claim-1}
For any $0<c<1$ ,
$$E(q)\leq 2c+\frac{4n-2}{R}G_L(c)+\epsilon,~\text{ if}~ \frac{4n-2}{R}G_L(1)>\epsilon.$$
\end{claim}

  Suppose that the claim is not true.  Then there exists  point $q_0\in  B_p(1)$ such that
 for some $c$,
  $$E(q_0)>2c+\frac{4n-2}{R}G(c)+\epsilon.$$
  We  consider
$$u(x)\,=\,\frac{4n-2}{R}G(d(q_0,x))-E(x)$$
in the annulus $A_{q_0}(c,L)$.  Clearly,
  $$\Delta u\geq 0.$$
   Note that we may assume that  $p \in A_{q_0}(c,1)$.  Otherwise we have $E(q_0)\leq E(p)+2c$. On the other hand,
    it is easy to see that  on the inner boundary $\partial B_{q_0}(c)$,
\begin{align}
\nonumber u(x)=\frac{4n-2}{R}G_L(c)-E(x)\leq \frac{4n-2}{R}G_L(c)-E(q_0)-2c\leq-\epsilon,
\end{align}
and on the outer boundary $\partial B_{q_0}(L)$,
\begin{align}
\nonumber u(x)=-E(x)\leq 0.
\end{align}
Thus applying  the maximum principle,  we obtain $ u(p)\le 0$.
However,
$$
u(p)=\frac{4n-2}{R}G_L(d(p,q_0))-E(p)\geq \frac{4n-2}{R}G_L(1)-\epsilon>0,
$$
which is impossible.  Therefore,  the claim is true.

Now if $R\epsilon \leq G_2(1)$, we choose $L=2$ and $c=(\frac{1}{R})^{\frac{1}{2n-1}}$, we have
$$E(q)\leq \epsilon+c(n)(\frac{1}{R})^{\frac{1}{2n-1}}.$$
Otherwise we choose $G_L(1)=\epsilon R, G_L(c)=Rc$ and get
$$E(q)\leq \epsilon+ c(n)\epsilon^{\frac{1}{2n-1}}.$$
The lemma is proved.
\end{proof}
$b^+(x)=d(q^+,x)- d(q^+,p)$ and let  $h^+$  be a harmonic function which satisfies
$$\triangle h^+=0,~ \text{in} ~B_p(1)\cap \mathcal R,$$
 with $h^+=b^+$ on $\partial B_p(1)\cap \mathcal R$. Then

 \begin{lem}\label{harmonic-estimate} Under the conditions in Lemma \ref{excess},
 we have
\begin{align}\label{c0-h}
\|h^+-b^+\|_{L^\infty(B_p(1))}\leq \Psi(1/R,\epsilon),\end{align}
\begin{align}\label{h-gradient}\frac{1}{\text{vol}(B_p(1))}\int_{B_p(1)\cap \mathcal R}|\nabla h^+-\nabla b^+|^2\text{dv} \leq \Psi(1/R,\epsilon),\end{align}
\begin{align}\label{hessian-integral}\frac{1}{\text{vol}(B_p(\frac{1}{2}))}
\int_{B_p(\frac{1}{2})\cap \mathcal R} |\rm{Hess }\,h^+|^2\text{dv}
\leq\Psi(1/R,\epsilon).
\end{align}
\end{lem}
\begin{proof}
Choose a point   $q$ in $\partial B_p(2)\cap \mathcal R$ and let $g=\phi(d(q,\cdot))$, where  $\phi(r)=r^{2-2n}$.
Then
\begin{align}
\Delta g=\varphi'\Delta r+\varphi''\geq\frac{2n-1}{r}\varphi' +\varphi'' =1,~\text{in}~ B_p(1)\mathcal R.
\end{align}
 It follows that
 \begin{align}
\nonumber \Delta (h^+-b^++\Psi(1/R,\epsilon)g) > 0,~\text{in}~ B_p(1)\cap\mathcal R.
\end{align}
Thus by  the maximum principle \ref{max-principle},  we get
$$h^+-b^+ \leq \Psi(1/R,\epsilon).$$
On the other hand, we have
$$\Delta ( -b^--h^++\Psi(1/R,\epsilon)g)  > 0,~\text{in}~ B_p(1),$$
where $b^-=d(q^-,x)- d(p,q^{-})$.  Since  $b^++b^-$ is small as long as $1/R$ and $\epsilon$ are small by Lemma \ref{excess}, by  the maximum principle, we also
 get
 $$h^+-b^+ >\geq-(b^++b^-)-\Psi(1/R,\epsilon) \geq -\Psi(1/R,\epsilon).$$

 For the second estimate (\ref{h-gradient}),  taking the cut-off function $\gamma_\eta$ for a Riemannian manifold with singularity, we have
\begin{align} &\int_{B_p(1)\cap \mathcal R}\gamma^2_\eta|\nabla h^+-\nabla b^+|^2d\text{v}\notag \\
&=
\int_{B_p(1)\cap \mathcal R}\gamma^2_\eta(h^+-b^+)(\triangle b^+-\triangle h^+)d\text{v}+
2\int_{B_p(1)\cap \mathcal R} (b^+-h^+)\langle \nabla h^+-\nabla b^+,\nabla \gamma_\eta\rangle \gamma_\eta d\text{v}
\notag\\
 & \leq \,\int_{B_p(1)\cap \mathcal R}\left(\Psi(1/R,\epsilon)\,\gamma^2_\eta|\triangle b^+|\,+\,\frac{1}{2}\gamma^2_\eta|\nabla h^+-\nabla b^+|^2d\text{v}\,+\, 2 (b^+-h^+)^2|\nabla \gamma_\eta|^2\right)\,d\text{v}.\notag
\end{align}
Thus
$$\int_{B_p(1)\cap \mathcal R}\gamma^2_\eta|\nabla h^+-\nabla b^+|^2d\text{v}
\leq \Psi(1/R,\epsilon)\int_{B_p(1)\cap \mathcal R}\gamma^2_\eta|\triangle b^+|d\text{v}+
 C\eta.$$
Now we see
\begin{align}
& \int_{B_p(1)\cap \mathcal R}\gamma^2_\eta|\triangle b^+|d\text{v}\notag\\
&\leq|\int_{B_p(1)\cap \mathcal R}\gamma^2_\eta\triangle b^+d\text{v}|+2 \text{sup}_ {B_p(1)}(\triangle b^+)\text{vol}(B_p(1))\notag\\
&\leq |\int_{B_p(1)\cap \mathcal R} div(\gamma_\eta^2\nabla b^+) d\text{v}|+2\int_{B_p(1)\cap \mathcal R}|\nabla \gamma_\eta|d\text{v}+C\text{vol}(B_p(1))\notag\\
&\leq  \text{vol}({\partial B_p(1)})+C\text{vol}(B_p(1))\,\leq\, C\text{vol}(B_p(1)).\notag
\end{align}
Here we used (\ref{mono-formula}) at the last inequality.  Then (\ref{h-gradient}) follows by letting $\eta\rightarrow 0$.

To get  (\ref{hessian-integral}), we  choose a cut-off function $\varphi$ supported in $B_p(1)$  as constructed in  Lemma \ref{cut-off}.
Since
\begin{align}\label{hess}
 \frac{1}{2}\Delta(|\nabla h^+|^2-|\nabla b^+|^2)=|\text{Hess }h^+|^2+\text{Ric}_g(\nabla h^+,\nabla h^+)\geq |\text{Hess }h^+|^2,
\end{align}
and $|\nabla h^+|$ is bounded in the support of $\phi$ by Proposition  \ref{gradient-esti}, for $u=|\nabla h^+|^2-|\nabla b^+|^2$, we have
$$\Delta u\geq C|\nabla u|^2.$$
By Lemma \ref{part1} and Lemma \ref{part2}, we have
$$\lim_{\eta\rightarrow 0}\int_{B_p(1)}\phi\gamma_\eta^2\Delta u=\int_{B_p(1)\cap \mathcal R}u\Delta\phi.$$
By (\ref{hess}), we derive (\ref{hessian-integral}) from (\ref{h-gradient}) immediately.
\end{proof}

\begin{lem}\label{triangle-lemma}
For any $\eta>0$, there exists $\delta=\delta(\eta)$ having the following property: let $x,y,z$ be three points in $B_p(1)\cap \mathcal R$ with $$|h^+(y)-h^+(x)-d(x,y)|\leq \delta, |h^+(x)-h^+(z)\leq \delta.$$
$\gamma(s)(s\in [0,c]))$  is
 the minimal geodesic  curve connecting $x,y$ and
 $\gamma_s(t)(s\in [0,l(s)], ~l(s)=d(z,\gamma (s)))$  is a family of minimal geodesic curves  connecting $z$ and $\gamma(s)$.
Assume that
\vskip 0.1in
\noindent
i) $|h^+\,-\,b^+|_{C^0(B_p(1))}\,\leq\,\delta$;
\vskip 0.1in
\noindent
ii) $\int_0^c|\nabla h^+(\gamma(s))-\nabla b^+(\gamma(s))|\,\leq\,\delta$;
\vskip 0.1in
\noindent
iii) $\int_0^c\int_0^{l(s)}|{\rm Hess}\, h^+\,(\gamma_s(t))|dtds\,<\,\delta$.
\vskip 0.1in
\noindent
Then
 \begin{align}\label{triangular-equ}
 |d(z,x)^2+d(x,y)^2-d(y,z)^2|<\eta.
 \end{align}
 \begin{proof}
 Since the rectangular is convex,  we can follow the proof of Lemma 9.16 in [Ch2].  From $|h^+(y)-h^+(x)-d(x,y)|\leq \delta$, by i) we know $|b^+(y)-b^+(x)-d(x,y)|\leq 3\delta$. Since $b^+$ is $1$-Lipschitz, we know that
   $$s= b^+(\gamma(s))-b^+(x)+\Psi(\delta).$$
Combined with $|h^+(x)-h^+(z)\leq \delta$, we get
\begin{align}\frac{1}{2} d(x,y)^2&=\int_0^c sds\notag\\
&=\int_0^c (h^+(\gamma(s))-h^+(x))ds+\Psi(\delta)\notag\\
&=\int_0^c (h^+(\gamma_s(l(s)))-h^+(\gamma_s(0)))ds+\Psi(\delta)\notag\\
&=\int_0^{l(s)}\int_0^c \langle\nabla h(\gamma_s(t)),\gamma_s'(t)\rangle dtds+\Psi(\delta).\notag
\end{align}
On the other hand,
\begin{align} &| \langle\nabla h(\gamma_s(t)),\gamma_s'(t)\rangle- \langle\nabla h(\gamma_s(l(s))),\gamma_s'(l(s))\rangle|\notag\\
&=|\int_t^{l(s)} \text{hess h}(\gamma_s'(\tau),  \gamma_s'(\tau)) d\tau|\notag\\
&\le \int_0^{l(s)} |\text{hess h}(\gamma_s'(t),  \gamma_s'(t))| dt.\notag
\end{align}
Hence from the condition iii),  we get
\begin{align}\label{short-distance}\frac{1}{2} d(x,y)^2&=\int_0^{l(s)}\int_0^a \langle\nabla h(\gamma_s(l(s))),\gamma_s'(l(s))\rangle dtds+\Psi(\delta)\notag\\
&=\int_0^a \langle\nabla h(\gamma_s(l(s))),\gamma_s'(l(s))\rangle l(s)ds+\Psi(\delta)\\
&=\int_0^a \langle\nabla  h(\gamma(s)),\gamma_s'(l(s))\rangle l(s)ds+\Psi(\delta).\notag
\end{align}
Since $|b^+(y)-b^+(x)-d(x,y)|\leq 3\delta$, we have
\begin{align}
\int_0^c|\nabla b^+-\gamma'(s)|ds&\leq \,c\int_0^c|\nabla b^+-\gamma'(s)|^2ds\notag\\
&=\,2 c \left(c-\int_0^c (b^+)'(\gamma(s))ds\right)\notag\\
&=\,2c[c-(d(q,y)-d(q,x))]\,\leq\, 6 c \delta.
\end{align}
Combined with ii) we get
\begin{align}\label{gra1}
&\int_0^c|\nabla h^+(\gamma(s))-\gamma'(s)|\notag\\
=&\int_0^c|\nabla h^+(\gamma(s))-\nabla b^+(\gamma(s))|ds+\int_0^c|\nabla b^+(\gamma(s))-\gamma'(s)|\,\leq\, 20\delta.
\end{align}

Now by  the first variation formula of geodesic curve,   we see that
$$
l'(s)=\langle\gamma_s'(l(s)),\gamma'(s)\rangle.$$
Then by  (\ref{gra1}), we obtain
\begin{align}
&\int_0^a \langle\nabla  h(\gamma(s)),\gamma_s'(l(s))\rangle l(s)ds\notag\\
&=\int_0^a  l'(s) l(s)ds +\Psi(\delta)\notag\\
&=\frac{1}{2} (d(y,z)^2-d(z,x)^2)+\Psi(\delta).\notag
\end{align}
Therefore, combined with (\ref{short-distance}), we derive (\ref{triangular-equ}).

 \end{proof}
\end{lem}
\begin{lem}\label{almostplanar1}
For any $\eta>0$, there exists $\delta=\delta(\eta)>0, R_0=R_0(\eta),\epsilon_0=\epsilon_0(\eta)$ having the following property: under the condition of Lemma \ref{excess}, if $R\geq R_0, \epsilon\leq \epsilon_0$, then for any three points $x,y,z$ satisfying
$$|h^+(y)-h^+(z)|\,\leq\, \delta, |h^+(x)-h^+(y)-d(x,y)|\,\leq \,\delta,$$
we have
$$|d^2(x,y)+d^2(y,z)-d^2(x,z)|\leq \eta.$$
\end{lem}
\begin{proof}
By Lemma \ref{segment-inequ1} and Lemma \ref{harmonic-estimate}, we get for any small positive $\eta_1$
 \begin{align}
 &\int_{B_x(\eta_1)\times B_y(\eta_1)\times B_z(\eta_1)}\int_{\Delta_{x^*y^*z^*}}|{\rm Hess }\,h^+|\notag\\
 \leq & C(n)\left(vol(B_x(\eta_1)\right)vol(B_p(1))\int_{B_p(1)}|{\rm Hess }h^+|\notag \\
 \leq &C(n)\left(vol(B_x(\eta_1)\right)vol(B_p(1))\Psi(\frac{1}{R},\epsilon).\notag
 \end{align}
So there are points $x^*\in B_x(\eta_1),y^*\in B_y(\eta_1),z^*\in B_z(\eta_1)$, such that
\begin{align}\nonumber
\int_0^{d(x^*,y^*)} ds\int_0^{d(z^*,\gamma(s))}|{\rm Hess }\,h^+(\gamma_s(t))|dt
 \leq\left(\frac{{\rm vol}(B_q(1))}{{\rm vol}(B_q(\eta_1))}\right)^2\Psi(\frac{1}{R},\epsilon),
\end{align}
where  $\gamma_s(t)$  is the minimal geodesic curves connecting $ \gamma(s)$  and $z^*$.
Now we have $$|h^+(y^*)-h^+(x^*)-d(x^*,y^*)|\leq 2C\eta_1+2\eta_1+\delta.$$
By Lemma \ref{triangle-lemma}, we know that there exist $\delta_0=\delta_0(\eta)>0, R_0=R_0(\eta),\epsilon_0=\epsilon_0(\eta),\eta_0=\eta_0(\eta)\leq \frac{\eta}{12}$ such that if $\delta\leq\delta_0, R\geq R_0, \epsilon\leq \epsilon_0, \eta_1\leq \eta_0$, we have
$$|d^2(x^*,y^*)+d^2(y^*,z^*)-d^2(x^*,z^*)|\leq \frac{\eta}{2}.$$
As a consequence, we get
$$|d^2(x,y)+d^2(y,z)-d^2(x,z)|\leq \frac{\eta}{2}+6\eta_1\leq\eta.$$
\end{proof}
\begin{lem}\label{product}
 Suppose that $X$ is a length space and $x^*$ is point in $X$. Assume that there is a function $h$ having the following two properties: \\
 i) $h$ is 1-Lipschitz with $h(x^*)=0$,\\
 ii) for any point $x\in B_{x^*}(1)$ and $t\in [-1,1]$, there exist $x_t\in X$ and a minimal geodesic $\gamma_t$ connecting $x$ and $x_t$ such that
 $$h(x_t)=t,d(x,x_t)=|h(x)-t|.$$
 iii) for three points $x,y,z\in B_{x^*}(1)$ with $h(x)=h(y), |h(x)-h(z)|=d(x,z)$, we have
$$d(y,z)^2=d(x,z)^2+d(x,y)^2.$$
Then there exists a metric space $Y$ such that
$$B_{x^*}(1)\cong B_{y^*\times 0}(1)\subset Y\times \mathbb{R}.$$
\end{lem}
\begin{proof}
Define $Y=h^{-1}(0)$ with the distance induced from $X$. For any $x\in B_{x^*}(1)$, by ii) there is a point $x_0\in Y$ such that $d(x,x_0)=|h(x)|$. We show that such point $x_0$ is unique. Assume that $x'_0$ is another point, then by iii) we have
$$|h(x)|^2=d(x,x_0)^2=d(x,x'_0)^2+d(x_0,x'_0)^2=|h(x)|^2+d(x_0,x'_0)^2.$$  It implies that
$x_0=x'_0$. Now denote $x_0$ by $\pi(x)$. For any two points $x,y\in B_{x^*}(1)$, assuming that $|h(y)|\geq |h(x)|$, we can choose a point with $h(z)=h(x)$ and $d(z,y)=h(y)-h(x)$.
We are going to show that $\pi(y)=\pi(z)$. We divide into two cases. The first case is that $h(y),h(z)$ have the opposite signs. Denoting the minimal geodesic connecting $y$ and $z$ by $\gamma(s)$, there is a point $w$ on $\gamma(s)$ with $h(w)=0$.
By i) we know that $$|h(y)-h(z)|=d(y,z)=d(y,w)+d(w,z)\geq |h(y)|+|h(z)|.$$
So we have $d(y,w)=|h(y)|$ which implies that $w=\pi(y)=\pi(z)$. For the second case,  denote the minimal geodesic connecting $y$ and $\pi(y)$ by $\gamma(s)$. There is a point $w$ on $\gamma(s)$ with $h(w)=h(x)$ and $d(y,w)=|h(y)-h(w)|$. By iii) we know that $d(w,z)=0$ which implies that $\pi(z)=\pi(w)=\pi(y)$. Now by iii) we also see that
\begin{align}
d(x,z)^2&=\,d(\pi(x),z)^2-|d(x,\pi(x))|^2\,=\,d(\pi(x),z)^2-|h(x)|^2\notag\\
&=\,d(\pi(x),z)^2-|h(z)|^2\,=\,
d(\pi(x),z)^2-|d(z,\pi(z))|^2\notag\\
&=\,d^2(\pi(x),\pi(z))=d^2(\pi(x),\pi(y)).\notag
\end{align}
By iii) we get $$d(x,y)^2=d(x,z)^2+|h(y)-h(x)|^2=d(\pi(x),\pi(y))^2+|h(y)-h(x)|^2.$$
It follows that $x\rightarrow (\pi(x),h(x))$ is a isometry. The lemma is proved.
\end{proof}
Now we are ready to prove Proposition \ref{split}.
\begin{proof}[Proof of Proposition \ref{split}]
Denote the line in $X$ by $\gamma(t)$ and $\gamma(0)=x^*$. Let $q_i^+,q_i^-\in M_i$ be the points converging to $\gamma(i),\gamma(-i)$ respectively such that
$$d(p_i,q_i^+)+d(p_i,q_i^-)-d(q_i^+,q_i^-)\leq \epsilon_i \rightarrow 0.$$ Denote by $h^+_i$ the functions constructed in Lemma. $h^+_i$ converges to a limit function $h$. By Lemma \ref{harmonic-estimate}, we know that
$$h(x)=\lim_{s\rightarrow +\infty}b_s(x),$$ where $b_s(x)=d(x,\gamma(s))-s$.
Denote $$h^-(x)=\lim_{t\rightarrow -\infty} d(x,\gamma(s))-s.$$ By Lemma \ref{excess}, we know that
$$h+h^-=0.$$
Now we show that $h$ satisfies the conditions in Lemma \ref{product}. i) is obvious.
For ii), let $x\in B_{x^*}(1)$ be any point. For $t\in [-1,h(x)]$, we choose a point $x^s_t$ on the minima geodesic connecting $x$ and $\gamma(s)$ with
$$d(x,x^s_t)=h(x)-t.$$ Then $b_s(x^s_t)=b_s(x)+t-h(x).$ Taking $s_i\rightarrow -\infty,$ we have $x^s_t\rightarrow z.$ Then we have
$$h(z)=t.$$ For $t\in [h(x),1]$, we can use $h^-(x)$ instead of $h(x)$ to obtain the point $z$.
For iii), let $x,y,z$ be three points in $B_{x^*}(1)$, with $h(x)=h(y), |h(x)-h(z)|=d(x,z)$. There are points $x_i,y_i,z_i\in B_{p_i}(1)$ converging to $x,y,z$ respectively such that
$$|h_i^+(x)-h^+(y)|\rightarrow 0, |h_i^+(x)-h^+(z)-d(x,z)|\rightarrow 0.$$ By Lemma \ref{almostplanar1}, we know that
$$|d(x_i,y_i)^2+d(x_i,z_i)^2-d(y_i,z_i)^2|\rightarrow 0,$$
which means that
$$d(x,y)^2+d(y,z)^2=d(y,z)^2.$$

\end{proof}
\section{Metric cone}
We define the following set of Riemannian manifold with singularity
$$\mathcal{M}(v,n)=\{(M^n,p,g)|{\rm Ric}(g)\,\geq\, 0 \text{ in }\mathcal R, vol(B_p(1))\geq v>0.\}.$$
Let $(M_i,p_i)$ converge to $(X,x)$ in the Gromov-Hausdorff sense. In this section, we prove that every tangent cone of $X$ is a metric cone:
 \begin{prop}\label{metriccone}
 Let $T_{x^*}X$ be a tangent cone at $x^*\in X$. Then there is a length space $Y$ such that
 $$T_{x^*} X\cong C(Y).$$
 \end{prop}
 The proof depends on the following lemmas. We start with some estimates of approximate harmonic functions. Let $(M^n,p,g)\in \mathcal{M}(v,n)$ and $q\in \mathcal R \subseteq M$ and  $h$ be a solution of the following equation,
\begin{align}\label{f-harmonic-radial}
\Delta h\,=\,n, ~\text{in}~ B_q(b)\setminus \overline{B_q(a)}, ~h|_{(\partial B_q(b)\cap \mathcal R)}\,=\,\frac{b^2}{2}~\text{ and}~   h|_{\partial (B_q(a)\cap \mathcal R)}\,=\,\frac{a^2}{2}.
\end{align}
   Let $p=\frac{r(q,\cdot)^2}{2}$.

   \begin{lem}\label{harmonic-estimate-annual-1}
 Suppose that
    \begin{align}\label{cone-volume-condition-2}
 \frac{{\rm vol}(\partial B_q(b))}{{\rm vol}(\partial B_q(a))}\geq(1-\omega)\frac{b^{n-1}}{a^{n-1}}
 \end{align}
for some $\omega>0$.
Then
 \begin{align}\label{gradient-annual}\frac{1}{{\rm vol}(A_q(a,b))}
 \int_{A_q(a,b)\cap \mathcal R}|\nabla p-\nabla h|^2  d\text{v} < \Psi(\omega;a,b).
 \end{align}
Moreover,
\begin{align}\label{h-p-estimate}
 \|h-p\|_{L^\infty(A_q(a',b'))}< \Psi(\omega;a,b,a',b'),
 \end{align}
 where   $a<a'<b'<b$.

 \end{lem}

 \begin{proof}  Since
 $$\Delta r\leq \frac{n-1}{r} \text{ in }\mathcal R,$$
 we have
 \begin{align}\label{lapalace-p}
  \Delta p=p''+p' \Delta r  \leq n, ~\text{ in }A(a,b)\cap \mathcal R.
   \end{align}
   Thus we get
\begin{align}\label{lemma-2-5-1}
 \frac{1}{\text{vol}(A_q(a, b))}\int_{A_q(a,b)\cap \mathcal R}\Delta pd\text{v}\leq n.
\end{align}
 On the other hand, by   the monotonicity formula (\ref{mono-formula}), we have
\begin{align}
\nonumber \text{vol}(A_q(a, b))\leq \frac{b^n-a^n}{na^{n-1}}\text{vol}(\partial B_q(a)).
\end{align}
It follows by (\ref{cone-volume-condition-2}),
$$\text{vol}(A_q(a, b))\leq  (1-\omega)^{-1}\frac{b^n-a^n}{nb^{n-1}}\text{vol}(\partial B_q(b)).$$
Since
\begin{align}
 \nonumber \int_{A_q(a,b)\cap \mathcal R}\Delta p d\text{v}=b\text{vol}(\partial B_q(b)) -a\text{vol}(\partial B_q(a)),
 \end{align}
we get
\begin{align}
 \nonumber &\frac{1}{\text{vol}(A_q(a, b))} \int_{A(a,b)\cap \mathcal R}\Delta p d\text{v}\\
 &\ge
  (1-\omega)\frac{nb^{n-1}}{b^n-a^n} \left( b-
a\frac{\text{vol}(\partial B_q(a))}{ \text{vol}(\partial B_q(b))}\right). \notag
 \end{align}
  Hence we derive immediately,
\begin{align}\label{lemma-2-5-2}
 \frac{1}{\text{vol}(A_q(a, b))}\int_{A_q(a,b)\cap \mathcal R}\Delta pd\text{v}\geq    n+\Psi(\omega;a, b)).
\end{align}

By (\ref{lemma-2-5-1}) and (\ref{lemma-2-5-2}), we have
\begin{align}\label{lap}
\int_{A_q(a,b)\cap \mathcal R}|\Delta p-n|  d\text{v}<  \text{vol}(A_q(a, b)) \Psi(\omega; a, b).
\end{align}
From
\begin{align}
0=&\int_{A_q(a,b)\cap \mathcal R}\gamma_\epsilon^2|\nabla (p-h)|^2+2\int_{A_q(a,b)\cap \mathcal R}(p-h)\gamma_\epsilon\langle \nabla \gamma_\epsilon,\nabla (p-h)\rangle\notag\\
&+\int_{A_q(a,b)\cap \mathcal R}(p-h)\gamma_\epsilon^2(\Delta p-\Delta h)\notag,
\end{align}
and
$$\int_{A_q(a,b)\cap \mathcal R}(p-h)\gamma_\epsilon\langle \nabla \gamma_\epsilon,\nabla (p-h)\rangle\leq C\left( \int_{A_q(a,b)\cap \mathcal R}|\nabla \gamma_\epsilon|^2\int_{A_q(a,b)\cap \mathcal R}\gamma_\epsilon^2|\nabla (p-h)|^2\right)^\frac{1}{2},$$
we get
$$\int_{A_q(a,b)\cap \mathcal R}\gamma_\epsilon^2|\nabla (p-h)|^2\leq C\int_{A_q(a,b)\cap \mathcal R}|\Delta p-n|+C\epsilon.$$
Then by (\ref{lap}), we obtain (\ref{gradient-annual}).

Applying the following Lemma \ref{Poincare1} to the function $p-h$ together with   the estimate (\ref{gradient-annual}),  we see that
\begin{align}
\nonumber \frac{1}{\text{vol}(A_q(a,b))}\int_{A_q(a,b)\cap \mathcal R}|p-h|^2  d\text{v} < \Psi(\omega; a,b).
\end{align}
Then  for any point $x\in A_q(a',b')\cap \mathcal R$, there is a point $y\in B_x(\eta)\cap \mathcal R$  such that
\begin{align}
|p(y)-h(y)|^2&\leq \frac{\text{vol}(A_q(a,b))}{\text{vol}(B_x(\eta))}  \frac{1}{\text{vol}(A_q(a,b))}\int_{A_q(a,b)}|p-h|^2  d\text{v}\notag\\
&< \frac{C(\Lambda, b)}{\eta^n}\Psi(\omega; a,b).\notag\end{align}
On the other hand,  by  Proposition \ref{gradient-esti}, we have
\begin{align} |(p(x)-h(x))-(p(y)-h(y))| &\leq (\|\nabla h\|_{C^0(A_q(a'-\eta,b'+\eta))}+1)\text{dist}(x,y)\notag\\
&\leq C(a,b,a'-\eta,b'+\eta)\eta.\notag\end{align}
Thus  we derive
\begin{align}
&|p(x)-h(x)|\notag\\
&< \frac{C(\Lambda, b)}{\eta^n}\Psi(\omega; a,b)+C(a,b,a'-\eta,b'+\eta)\eta.\notag
\end{align}
Choosing $\eta=\Psi^{\frac{1}{n+1}}$, we  prove  (\ref{h-p-estimate}).
\end{proof}
\begin{lem}\label{Poincare1}
Let $f\in L^\infty(A_q(a,b))$ be a locally Lipschitz function in $A_q(a,b)\bigcap \mathcal R$ and $f|_{\partial A_q(a,b)\cap \mathcal R}=0$, then there is a positive number $\lambda_1\leq C(b,n)$ such that
$$\lambda_1\int_{A_q(a,b)\cap \mathcal R} f^2\leq \int_{A_q(a,b)\cap\mathcal R} |\nabla f|^2.
$$
\end{lem}
\begin{proof}
As in the proof of Lemma \ref{excess}, let $G_b$ be the function satisfying $\Delta G_b\geq 1.$
We have
$$\int_{A_q(a,b)\cap\mathcal R} f^2\leq \int_{A_q(a,b)\cap\mathcal R} f^2\Delta G_b.$$
Let $\gamma_\eta$ be the cut-off function, then we have
\begin{align}
\int_{A_q(a,b)\cap \mathcal R} f^2\gamma_\eta\Delta G_b&=-\int_{A_q(a,b)\cap \mathcal R}2f\langle \nabla f,\nabla G_b\rangle\gamma_\eta-\int_{A_q(a,b)\mathcal R}f^2\langle \nabla G_b,\nabla \gamma_\eta\rangle.\nonumber \\
&\leq C(b,n)\left(\int_{A_q(a,b)\cap \mathcal R} f^2\right)^{\frac{1}{2}}\left(\int_{A_q(a,b)\cap \mathcal R} |\nabla f|^2\right)^{\frac{1}{2}}+C\eta.\nonumber
\end{align}
Taking $\eta \rightarrow 0$, we get
$$\int_{A_q(a,b)\cap \mathcal R} f^2\Delta G_b\leq C(b,n)\int_{A_q(a,b)\mathcal R} |\nabla f|^2.$$
It follows that
$$\int_{A_q(a,b)\cap\mathcal R} f^2\leq C(b,n)\int_{A_q(a,b)\mathcal R} |\nabla f|^2.$$
\end{proof}

Furthermore,  we have
\begin{lem}\label{harmonic-estimate-annual-2} Under the condition in Lemma \ref{harmonic-estimate-annual-1}, it holds
 \begin{align}&\frac{1}{{\rm vol}(A_q(a',b'))}\int_{A_q(a',b')\cap \mathcal R}|{\rm Hess }\,h-g|^2d\text{v}\notag\\
 &< \Psi(\omega;a,b,a',b'),
 \end{align}
  where  $a<a'<b'<b$.
   \end{lem}

 \begin{proof} First observe that
 \begin{align}
 |\text{Hess }h-g|^2=|\text{Hess h}|^2+(n-2\Delta h).\notag
 \end{align}
Let  $\varphi$ be a cut-off function  of $A_q(a,b)$ as constructed  in Lemma \ref{cut-off} which satisfies,
\begin{align}
& 1)   ~\varphi\equiv 1,~\text{ in}~ A_q(a',b')\cap \mathcal R;\notag\\
&  2) ~|\nabla\varphi|,  |\triangle\varphi|~\text{is bounded in}~ A_q(a,b)\cap\mathcal R.\notag
\end{align}
 Then
 \begin{align}\label{two-parts}
 \nonumber &\frac{1}{\text{vol}(A_q(a,b))}\int_{A_q(a,b)\cap \mathcal R}\varphi|\text{Hess h}-g|^2 d\text{v}\\
 \nonumber &=\frac{1}{\text{vol}(A_q(a,b))}\int_{A_q(a,b)\cap \mathcal R}\varphi|\text{Hess }h|^2 d\text{v}\\
 &-\frac{1}{\text{vol}(A_q(a,b))}\int_{A_q(a,b)\cap \mathcal R}n\varphi d\text{v}.
 \end{align}

 By  the  Bochner formula Lemma \ref{part1} and Lemma \ref{part2}, we have
 \begin{align}\label{a1}
 \nonumber  \frac{2}{\text{vol}(A_q(a,b))}\int_{A_q(a,b)\cap \mathcal R}\varphi|\text{Hess }h|^2 d\text{v}&\leq \frac{1}{\text{vol}(A_q(a,b))}\int_{A_q(a,b)\cap \mathcal R}\varphi \Delta |\nabla h|^2d\text{v}\\
 &=\frac{1}{\text{vol}(A_q(a,b))}\int_{A_q(a,b)\cap \mathcal R} |\nabla h|^2 \Delta \varphi d\text{v}
 \end{align}

By Lemma \ref{harmonic-estimate-annual-1} and Lemma \ref{part1}, we have
\begin{align}\label{a2}
\nonumber  \frac{1}{\text{vol}(A_q(a,b))}\int_{A_q(a,b)}|\nabla h|^2 \Delta\varphi d\text{v}
&\leq \frac{1}{\text{vol}(A_q(a,b))}\int_{A_q(a,b)}|\nabla p|^2\Delta \phi d\text{v}+\Psi(\omega;a,b,a',b')\\
 &=\frac{2}{\text{vol}(A_q(a,b))}\int_{A_q(a,b)} p\Delta\varphi d\text{v}+\Psi(\omega;a,b,a',b') \notag\\
 &=\frac{2}{\text{vol}(A_q(a,b))}\int_{A_q(a,b)} \phi \Delta p d\text{v}+\Psi(\omega;a,b,a',b')
 \end{align}
It follows from (\ref{two-parts}), (\ref{a1}) and (\ref{a2}),
 \begin{align}
 \nonumber &\frac{1}{\text{vol}(A_q(a_1,b_1))}\int_{A_q(a_1,b_1)\cap \mathcal R}|\text{Hess }h-g|^2 d\text{v}\\
 \nonumber & \leq\frac{C(a_1,b_1,a,b)}{\text{vol}(A_q(a,b))}\int_{A_q(a,b)\cap \mathcal R}\varphi|\text{Hess }h-g|^2 d\text{v}\\
 \nonumber &\leq\frac{C(a_1,b_1,a,b)}{\text{vol}(A_q(a,b))}\int_{A_q(a,b)\cap \mathcal R}\varphi(\Delta p-n) d\text{v}+ \Psi(\omega;a,b,a',b')
 \\
 \nonumber &\leq  \Psi(\omega;a,b,a',b').
 \end{align}
Here we used  (\ref{lapalace-p}) at last inequality.
\end{proof}

\begin{lem}\label{planar}
 Given $b>a>0$, for any $\epsilon>0$, there exits $\delta>0$ such that the following holds: let $x,y\in A_q(a,b)$ be two points with $d(x,y)\leq r(y)-r(x)+\delta$, $\gamma(s)(0\leq s \leq c)$ be the unique geodesic connecting $x$ and $y$, and $\gamma_s(t)$  be  a family of geodesic  curves  connecting $z$ and $\gamma(s)$.
Suppose that
\begin{align}
&i)~~ |h-p|_{C^0(A_q(a,b))}\leq \delta; \notag \\
&ii)~~ \int_0^c|\nabla h(\gamma(s))-\nabla p|\leq\delta;\notag\\
& iii)~~\int_0^c\int_0^{l(s)}|{\rm Hess }\,h-g|dtds \leq\delta\notag
\end{align}
where $r(\cdot)=\text{dist}(q,\cdot)$. Then
\begin{align}\label{cosine}
|d(z,y)^2r(x)-d(x,z)^2r(y)+ r(z)^2\big(r(y)-r(x)\big)-r(x)r(y)\big(r(y)-r(x)\big)| <\epsilon.
\end{align}
\end{lem}
\begin{proof}
From  $d(x,y)\leq r(y)-r(x)+\delta$, we know that $|r(\gamma(s))-(r(x)+s)|\leq \delta.$ Then by i) we get
\begin{align}\label{diff}
|h(\gamma(s))-\frac{(r(x)+s)^2}{2}|\leq C(b)\delta.
\end{align}
From
\begin{align} h(\gamma_s(l(s)))-h(z)&=\int_0^{l(s)}\frac{d}{dt}h(\gamma_s(t)),\notag\\
\frac{d}{dt}h(\gamma_s(t))|_{t=l(s)}-\frac{d}{dt}h(\gamma_s(t))&=\int_t^{l(s)}\text{Hess }h(\gamma_s'(\tau),\gamma_s'(\tau)\notag,
\end{align}
we have
\begin{align} l(s) h'(\gamma_s(l(s)))&=h(\gamma_s(l(s)))-h(z)-\frac{l^2(s)}{2}\notag\\
&+\int_0^{l(s)} dt\int_t^{l(s)} (\text{Hess }h(\gamma_s'(\tau),\gamma_s'(\tau)) -g(\gamma_s'(\tau),\gamma_s'(\tau)))  d\tau.\notag
\end{align}

From (\ref{diff}), we get
 \begin{align}l(s)h'(\gamma_s(l(s)))&=\frac{(r(x)+s)^2-r^2(z)}{2}+\frac{l^2(s)}{2}\notag\\
&+ \int_0^{l(s)} dt\int_t^{l(s)} (\text{Hess }h(\gamma_s'(\tau),\gamma_s'(\tau)) -g(\gamma_s'(\tau),\gamma_s'(\tau)))  d\tau+
 \Psi(\delta).\notag
\end{align}

Hence we derive
\begin{align}
\int_0^c\left(\frac{2l(s)h(\gamma_s'(l(s)))}{(s+r(x))^2}-\frac{l^2(s)}{(s+r(x))^2}\right) ds=a+\frac{r^2(z)}{r(x)+c}-\frac{r^2(z)}{r(x)}\notag
\\+\int_0^c \frac{ds}{(r(x)+s)^2}\int_0^{l(s)} dt\int_t^{l(s)} (\text{Hess }h(\gamma_s'(\tau),\gamma_s'(\tau)) -g(\gamma_s'(\tau),\gamma_s'(\tau)))  d\tau+ \Psi(\delta).\notag
\end{align}
By iii), we get
\begin{align}\label{short-distance-2}
\int_0^a\left(\frac{2l(s)h'(\gamma_s(l(s)))}{(s+r(x))^2}-\frac{l^2(s)}{(s+r(x))^2}\right) ds=c+\frac{r^2(z)}{r(x)+c}-\frac{r^2(z)}{r(x)}+\frac{\Psi(\delta)}{r(x)^2}.
\end{align}
Since $d(x,y)\leq r(y)-r(x)+\delta,$ we have
\begin{align}
\int_0^c|\nabla d(q,.)-\gamma'(s)|ds&\leq c\int_0^c|\nabla d(q,.)-\gamma'(s)|^2ds=2c\left(c-\int_0^c d'(q,\gamma(s))ds\right)\notag\\
&=2c[c-(d(q,y)-d(q,x))]\leq 2c\delta.
\end{align}
Combined with ii) we get
\begin{align}\label{gra}
&\int_0^c|\nabla h(\gamma(s))-r(\gamma(s))\gamma'(s)|\\
=&\int_0^c|\nabla h(\gamma(s))-\nabla p|ds+\int_0^c|\nabla p-r(\gamma(s))\gamma'(s)|\,\leq\, C(b)\delta.\notag
\end{align}
 By the first variation formula,
$$
l'(s)=\langle\gamma_s'(l(s)),\gamma'(s)\rangle,
$$
Now from (\ref{gra}) we get
\begin{align}
&\int_0^c\left(\frac{l^2(s)}{s+r(x)}\right)' ds\notag\\
&=\int_0^c\left(\frac{2l(s)l'(s)}{s+r(x)}-\frac{l^2(s)}{(s+r(x))^2}\right) ds\notag\\
&=\int_0^c\left(\frac{2l(s) (s+r(x))\langle \gamma_s'(l(s)),\gamma'(s)\rangle}{(s+r(x))^2}-\frac{l^2(s)}{(s+r(x))^2}\right)ds \notag\\
&=\int_0^c\left(\frac{2l(s)\langle\gamma_s'(l(s)),\nabla h(\gamma(s))\rangle}{(s+r(x))^2}-\frac{l^2(s)}{(s+r(x))^2}\right)ds+\Psi(\delta)\notag\\
  &=\int_0^ac \left(\frac{2l(s)h'(\gamma_s(l(s)))}{(s+r(x))^2} - \frac{l^2(s)}{(s+r(x))^2}\right) ds +\Psi(\delta)\notag.
   \end{align}
Combined with (\ref{short-distance-2}), we get (\ref{cosine})  immediately.

\end{proof}

\begin{lem}
Given $b>\epsilon>0$, there exits $\delta>0$ such that the following holds: assume that $x,y\in A_q(\epsilon,b)$ with $d(x,y)\leq r(y)-r(x)+\delta$ and $h$ satisfying
\begin{align}
&i)~ |h-p|_{C^0(A_q(\epsilon,b))}\leq \delta; \notag \\
&ii)~ \int_{A_q(\frac{\epsilon}{8},b)}|\nabla h-\nabla p|<\delta<<1;\notag\\
& iii)~\int_{A_q(\frac{\epsilon}{8},b)}|{\rm Hess }\,h-g|dtds <\delta<<1.\notag
\end{align}
Then for any $z\in A_q(\epsilon,b)$, we have
\begin{align}\label{cosine1}
|d(z,y)^2r(x)-d(x,z)^2r(y)+ r(z)^2\big(r(y)-r(x)\big)-r(x)r(y)\big(r(y)-r(x)\big) |\,<\,\epsilon.
\end{align}
\end{lem}
\begin{proof}
From Lemma \ref{harmonic-estimate-annual-2} and , we know that there exists $x'\in B_x(\eta), y'\in B_y(\eta), z'\in B_z(\eta)$, such that
\begin{align}
\int_{\gamma_{x'y'}}|\nabla h(\gamma(s))-\nabla p|\chi_{A_q(\frac{\epsilon}{8},b)}\leq \delta;\notag\\
\int_{\Delta_{x',y',z'}}|{\rm Hess }\,h-g|\chi_{A_q(\frac{\epsilon}{8},b)}\leq \delta \notag
\end{align}
If the geodesics $\gamma_s(t)$ all lies in $A_q(\frac{\epsilon}{8},b)$, the result follows from Lemma \ref{planar}. Now, we define
$$s_0=\sup\{s|d(q,\gamma_s)\leq \frac{\epsilon}{8}\}.$$ Denoting $\gamma(s_0)$ by $w$, we have $$|d(w,z')-(r(w)+r(z'))|\leq \frac{\epsilon}{4}.$$ It follows that
$$d(z',x')\geq d(z',w)-d(x',w)\geq r(w)+r(z')-\frac{\epsilon}{4}-s_0.$$
Since $|r(w)-r(x')-s_0|\leq \delta+\frac{1}{10}\epsilon$, we have
$$d(z',x')\geq r(x')+r(z')-\frac{\epsilon}{2}-\delta.$$

Applying Lemma \ref{planar} to $w,y',z'$,  we have
$$|d(z',y')^2r(w)-d(w,z')^2r(y')+ r(z')^2(r(y')-r(w))-r(w)r(y')(r(y')-r(w))|=\Psi(\delta|b).$$

The lemma is proved.
\end{proof}

\begin{lem}\label{almostplanar}
For any $\eta>0$, there is $\omega=\omega(\eta,a,b)$ such that the following holds:
if $$\frac{{\rm vol}(\partial B_p(b))}{{\rm vol}(\partial B_p(\omega))}\geq(1-\omega)\frac{b^{n-1}}{\omega^{n-1}},$$ and
$x,y,z \in A_p(a,b)$ satisfies $d(x,y)\leq r(y)-r(x)+\omega$, then we have
$$ d(z,y)^2r(x)-d(x,z)^2r(y)+r(z)^2(r(y)-r(x))-r(x)r(y)(r(y)-r(x)) <\eta. $$
\end{lem}

\begin{lem}\label{distance-appro-3}
Given $a<c<b$. For any $\eta>0$, there exists $\omega=\omega(a,b,c,\eta,n)$ such that the following is true:
if
\begin{align}\label{volume-condition-2}
\frac{{\rm vol}(\partial B_p(b))}{{\rm vol}(\partial B_p(a))}\geq(1-\omega)\frac{b^{n-1}}{a^{n-1}},
\end{align}
then  for any point $q$ on $\partial B_p(c)$, there exists $q'$ on $\partial B_p(b)$ such that
\begin{align}\label{distance}
d(q,q')\leq b-a+\eta.
\end{align}
 \end{lem}
 \begin{proof}
 We prove by contradiction. If there is point $q\in \partial B_p(c)$ such that $$d(q,x)\geq b-c+\eta, (\forall x\in \partial B_p(b)),$$ we can choose another point $q_1\in \mathcal{R}\bigcap B_q(\frac{\eta}{3})$ such that
 $$d(q,x)\geq b-c+\frac{\eta}{3}, (\forall x\in \partial B_p(b)).$$
 Since the $\mathcal{R}$ is convex, we know that every minimal geodesic connecting $p$ and $x\in \partial B_p(b)$ has no intersection with $B_{q_1}(\frac{\eta}{6})$. Then there is  some $\frac{\eta}{4}<r<\frac{\eta}{3}$ such that
\begin{align}
\text{vol} B_{q_1}(\frac{\eta}{3})\cap S_p(c+r)\geq c(n){\eta}^{n-1}\text{vol}A_p(a,b).
\end{align} Using  the monotonicity formula  (\ref{mono-formula}),  we get
\begin{align}
\text{vol}S_p(b)\notag &\leq \text{vol}(S_p(a+r)\setminus B_{q_1}(\frac{\eta}{3}))\frac{b^{n-1}}{(a+r)^{n-1}}\notag\\
&\leq (\text{vol}S_p(a+r)-c(n){\eta}^{n-1}\text{vol}A_p(a,b))\frac{b^{n-1}}{(a+r)^{n-1}},
\end{align}
 which is a contradiction.

 \end{proof}

\begin{lem}\label{cone}
Suppose $X$ is a length space and $x^*$ is a point in $X$. Assume that for any $x\in B_{x^*}(1)$ there exists $y\in\partial B_{x^*}(1)$ and a minimal geodesic $\gamma(t)$ from $x^*$ and $y$ containing $x$. Moreover, we assume that for any four points $y_1,y_2,z_1,z_2$ with
$$d(x^*,z_i)=d(x^*,y_i)+d(y_i,z_i) (1\leq i \leq 2),$$
we have
\begin{align}\label{cosineidentity}
\frac{d^2(x^*,y_1)+d^2(x^*,y_2)-d^2(y_1,y_2)}{d(x^*,y_1)d(x^*,y_2)}=
\frac{d^2(x^*,z_1)+d^2(x^*,z_2)-d^2(z_1,z_2)}{d(x^*,z_1)d(x^*,z_2)}.
\end{align}
Then there exists a metric space $Y$ such that
$$B_{x^*}(1)\cong B_o(1)\subset B_o(C(Y)).$$

\end{lem}
\begin{proof}
Let $Y$ be the set of all minimizing geodesics $\gamma: [0, 1]\rightarrow X$ with $\gamma(0)=x^*$. Then we can check the isometry directly.
\end{proof}

To get the condition of almost volume, we use the following lemma.
\begin{lem}\label{volumecone}
Given $0<a<b$, for any $\omega>0$, there exists $N=N(b,v,\omega)$, such that for any sequence of $r_i (1\leq i \leq N)$ satisfying $ar_i\geq b r_{i+1}$ and $(M,p)\in \mathcal{M}(n,k,v)$, there is some $j\in [1,N]$ such that
$$\frac{vol(A_p(br_j))}{vol(A_p(ar_j))}\geq (1-\omega)\frac{b^{n-1}}{a^{n-1}}.$$
\end{lem}
\begin{proof}[Proof of Proposition \ref{metriccone}]
We need to verify the conditions in Lemma \ref{cone} for $$(T_{x^*}X, \omega_x, x^*)=\lim_{i\rightarrow \infty}(X,\frac{d}{r_i^2},x^*).$$ Let $x\in B_{x^*}(1,\omega_{x^*})$ with $d(x^*,x)=2a>0$. For any $\epsilon>0$, let $\omega=\omega(\epsilon)$ be the constant determined in Lemma \ref{distance-appro-3}. By Lemma \ref{volumecone}, we know that there is a subsequence of $r_{j_i}\rightarrow 0$ such that
$$\frac{vol(A_p(r_{j_i}))}{vol(A_p(ar_{j_i}))}\geq (1-\omega)\frac{1}{a^{n-1}}.$$
So by Lemma \ref{distance-appro-3}, there is a point $q_i\in B_{p_i}(1)$ with $d(q_i,p_i)\leq 1-a+\eta$. Denoting the limit of $q_i$ by $y_\eta$, we have $1-a\leq d(x,y_\eta)\leq 1-a+\eta$. Since $\eta$ is arbitrary, we will find a point $y$ with $d(x,y)=1-a.$ Moreover, applying Lemma \ref{almostplanar} to $y_1,z_1,z_2 $, and $y_1,y_2,z_2 $, we know that
$$d^2(z_1,z_2)r(y_1)+d^2(x^*,z_2)(r(z_1)-r(y_1))-d^2(y_1,z_2)r(z_1)=r(y_1)r(z_1)(r(z_1)-r(y_1))$$
and
$$d^2(y_1,z_2)r(y_2)+d^2(x^*,y_1)(r(z_2)-r(y_2))-d^2(y_1,y_2)r(z_2)=r(y_2)r(z_2)(r(z_2)-r(y_2)).$$
Combined these two identity, we get (\ref{cosineidentity}).
\end{proof}
\section{Volume convergence}
In this section, we will prove a local version of volume convergence as in \cite{Co}.
Let $M$ be a Riemannian manifold with singularity and ${\rm Ric}(g)\,\geq\, 0$ in $\mathcal R$.
\begin{prop}\label{volume-estimate}
Given $\epsilon>0$, there exist $R=R(\epsilon,n)>1$  and $\delta=\delta(\epsilon,n)$ such that if $B_p(R)\subset M$ satisfies

 \begin{align}{\rm d}_{GH}(B_p(R),B_0(R))<\delta,\end{align}
  then we have
\begin{align}
{\rm vol}(B_p(1))\geq (1-\epsilon){\rm vol}(B_0(1)).
\end{align}
\end{prop}

\begin{proof}
We need to construct a Gromov-Hausdorff approximation  map by  using harmonic  functions  constructed  in Section 2.  Choose
$ n $ points $q_i$ in $B_p(R)$  which is close to $Re_i$ in $B_0(R)$, respectively.
Let  $l_i(q)=d(q,q_i)-d(q_i,p)$ and $h_i$ a solution of
 \begin{align}\Delta h_i=0, ~\text{in}~B_1(p), \notag\end{align}
with  $ h_i=l_i$  on  $\partial B_1(p)\cap \mathcal R$.
Then by Lemma \ref {harmonic-estimate},  we have
\begin{align}
 \nonumber \frac{1}{\text{vol }(B_p(1))}\int_{B_p(1)}|\text{Hess }h_i|^2 <\Psi(1/R,\delta).
\end{align}
By using  an  argument  in \cite{Co},  it follows
\begin{align}\label{orthorgonal-4}
\frac{1}{\text{vol }(B_p(1))}\int_{B_p(1)\cap \mathcal R}|\langle\nabla h_i,\nabla h_j\rangle-\delta_{ij}| <\Psi(1/R,\delta).
 \end{align}
Define a map by  $h=(h_1,h_2,...,h_n)$. It is easy to see that  the map $h$ is a $\Psi(\frac{1}{R}, \delta)$ Gromov-Hausdorff approximation to $B_p(1)$  by using the estimate (\ref{c0-h})  in Lemma \ref{harmonic-estimate}.
  Since $h$
 maps  $\partial B_p(1)$  nearby  $\partial B_0(1)$ with distance  less than $\Psi$,  by a small modification  to $h$ we may assume that
 \begin{align}
  \nonumber h:(B_p(1),\partial B_p(1))\longrightarrow(B_0(1-\Psi),\partial B_0(1-\Psi)).
  \end{align}

Now we can use the same degree  argument  in \cite{Ch}  to show that the image of $h$ contains $B_0(1-\Psi)$ .
  By using Vitali covering lemma,   there exists a point $x$ in $B_p(\frac{1}{8})\cap \mathcal R
  $ such that for any $r$  less than $\frac{1}{8}$ it holds
  \begin{align}\label{almost-hassian-zero-2}
  \frac{1}{\text{vol }(B_x(r))}\int_{B_x(r)\cap \mathcal R}|\text{Hess }h_i|<\Psi
  \end{align}
  and
  \begin{align}\label{orthogonal-2}
   \frac{1}{\text{vol }(B_x(r))}\int_{B_x(r)\cap \mathcal R}|\langle\nabla h_i,\nabla h_j\rangle-\delta_{ij}|<\Psi.
  \end{align}
    Let  $\eta=\Psi^{\frac{1}{2n+1}}$.  For any $y$ with  $d(x,y)=r<\frac{1}{8}$,
  applying Lemma \ref{segment-inequ}, we get from (\ref{almost-hassian-zero-2}),
  \begin{align}
&\int_{B_x(\eta r)\times B_y(\eta r)}\int_{\gamma_{zw}}|\text{Hess }h_i(\gamma',\gamma')|\notag\\
 &< r\big(\text{vol }(B_x(\eta r))+\text{vol }(B_y(\eta r))\big)\text{vol }(B_x(r)) \Psi\notag.&
  \end{align}
It follows that
 \begin{align}
&\int_{B_x(\eta r)\cap \mathcal R} \left(  Q(r,\eta) \int_{ B_y(\eta r)\cap \mathcal R}  \int_{\gamma_{zw}} \Sigma_{i=1}^n |\text{Hess }h_i(\gamma',\gamma')|
+  | \langle\nabla h_i,\nabla h_j\rangle-\delta_{ij}|\right)
\notag\\
 &<\text{vol }   (B_x(\eta r)) \Psi\notag,&
  \end{align}
 where $Q(r,\eta)=\frac{\text{vol}B_x(\eta r)}{ r\big(\text{vol }(B_x(\eta r))+\text{vol }(B_y(\eta r))\big)\text{vol}B_x(r)}$.
  Consider
  $$  Q(r,\eta) \int_{ B_y(\eta r)\cap \mathcal R} \int_{\gamma_{zw}}\Sigma_{i=1}^n\text{Hess } |h_i(\gamma',\gamma')| + | \langle\nabla h_i,\nabla h_j\rangle-\delta_{ij}|$$
  as a function of $z\in B_x(\eta r)$.   Then  one sees that  there exists a  point $x^*\in B_x(\eta r)\cap \mathcal R$ such that
 \begin{align}\label{orthogonal-3}
 |\langle\nabla h_i,\nabla h_j\rangle(x^*)-\delta_{ij}|< \Psi
 \end{align}
and
 \begin{align}\label{hession-small-geodesic}
\Sigma_{i=1}^n\int_{B_y(\eta r)\cap \mathcal R
} \int_{\gamma_{x^*w}}|\text{Hess }h_i(\gamma',\gamma')| < r\text{vol }(B_x(r))\eta^{-n}\Psi.
 \end{align}
 Moreover  by  (\ref{hession-small-geodesic}),   we  can find  a point  $y^*\in B_y(\eta r)$ such that
\begin{align}\label{almost-hessian-zero-3}
  \Sigma_{i=1}^n\int_{\gamma_{x^*y^*}}|\text{Hess }h_i(\gamma',\gamma')| <\eta r.
\end{align}

By a direct calculation with help of   (\ref{orthogonal-3}) and (\ref{almost-hessian-zero-3}),    we get
 \begin{align}\label{dist-appr-1}
 (h(x^*)-h(y^*))^2=(1+\Psi^{\frac{1}{2n+1}})r^2.
 \end{align}
 This shows that $ h(x)\neq h(y)$ for any $y$ with $d(y,x)\leq \frac{1}{8}$.  On the other hand, for any $y$ with $d(y,x)\geq \frac{1}{8}$,  it is clear that $h(x)\neq h(y)$ since  $h$ is a $\Psi$ Gromov-Hausdorff approximation.
 Thus  we prove that   the pre-image of $h(x)$ is unique.  Therefore the degree of $h$ is  $1$,  and consequently,
  $B_0(1-\Psi)\subset h(B_p(1))$.     The lemma  is proved  because the volume of $B_p(1)$ is almost same to one of $h(B_p(1))$
  by  (\ref{orthorgonal-4}).
\end{proof}
\section{Structure of limit spaces}
\subsection{Real case}
Let $(M_i,g_i)$ be a sequence of Riemannian manifolds with singularity in $\mathcal M(V,D,n)$ and $(M_i,g_i)\rightarrow (X,d)$.
\begin{thm}
Every tangent cone of $X$ is a metric cone. There is a decomposition of $X$ into $\mathcal R\cup \mathcal S$ and $\mathcal S=\mathcal S_{2n-2}$. Moreover, we have $\dim \mathcal S_k\leq k$.
\end{thm}
\begin{proof}
For any $(M,g)\in \mathcal M(V,D,n)$, by the volume comparison, we have
$$\frac{vol(B_p(1))}{V}\geq \frac{vol(B_p(1))}{\mathcal H^n(M)}\geq \frac{1}{D^n}.$$
Now by Proposition \ref{metriccone}, we know that every tangent cone is a metric cone. By the argument in \cite{CC2}, we get $\mathcal S=\mathcal S_{2n-2}$ and $\dim \mathcal S_k\leq k$.
\end{proof}
From this theorem and Proposition \ref{volume-estimate}, we have
\begin{prop}
Denote by $\mathcal H^n$ the $n$-dimensional Hausdorff measure, then
$$\lim_{i\rightarrow \infty} \mathcal H^n(M_i)=\mathcal H^n(X).$$
\end{prop}
\subsection{K\"ahler case}
Now let $M^n$ be a K\"ahler manifold, $\omega$ be a conic K\"{a}hler-Einstein metric on $M$:
$${\rm Ric}(\omega)\,=\,t\omega\,+\,2\pi\,\sum_{i=1}^k(1-\beta^i)D_i,$$
where $t$ is a positive constant and $D_i$ are simple normal crossing divisors in $M$. For any $\delta>0, V>0$, denote by $\mathcal M(n,k,\delta,V)$ the following set:
\begin{align}
\{(M^n,\omega)\,|\,\omega \text{a conic K\"{a}hler-Einstein metric with } t\in [\delta, \delta^{-1}], \int_M\omega^n\geq V\} \notag
\end{align}
\begin{lem}
Any manifold $M$ with a conic K\"ahler-Einstein metric $\omega$ is a Riemannian manifold with singularity.
\end{lem}
\begin{proof}
The convexity in Definition \ref{manifold} follows from Theorem 1.1 in \cite{D}. The existence of cut-off function is standard. To get the solution of Dirichlet problem, we use the approximation of $\omega$ by smooth K\"ahler metrics with Ricci curvature bounded from below. By Proposition 1.1 in \cite{D}  or Theorem 2.1 in \cite{TW}, there is a sequence of smooth K\"ahler metrics $\omega_i$ satisfying $\omega_i \rightarrow \omega$ smoothly outside $D=\bigcup_{i=1}^k D_i$ and ${\rm Ric}(\omega_i)\,\geq\, - C\,\omega_i$ for some constant depending on $(M,\omega)$. Let $h_i$ be the solution of
\[ \left \{ \begin{array}{l}
\Delta_{\omega_i} h_i = 0,\\
h_i|_{\partial U}=b|_{\partial U}.
\end{array} \right. \]
By Lemma \ref{gradient-esti} , we know that for any $V\subset\subset U$, $h_i$ is  uniformly Lipschitz on $V$. For $\Omega\subset\subset \overline U\bigcap \mathcal R$, there exists a constant $C=C(\Omega)$ such that $|\nabla h_i|_{\Omega}\leq C$. So we can take the limit of $h_i$ to get $h$.
\end{proof}

\begin{thm}
For any limit space $X$ of conic K\"ahler-Einstein metrics in $\mathcal M(n,k,\delta,V)$,
every tangent cone of $X$ is a metric cone. There is a decomposition of $X$ into $\mathcal R\cup \mathcal S$ and $\mathcal S=\mathcal S_{2n-2}$. Moreover,
$ \mathcal S_{2k+1}=\mathcal S_{2k}$ and $\dim \mathcal S_{2k}\leq 2k$.
\end{thm}
\begin{proof}
Since $M\setminus D$ is convex, we know that $diam(M,\omega)\leq \sqrt{\frac{2n-1}{\delta}}$. By the above lemma, we know that $$\mathcal M(n,k,\delta, V)\subseteq\mathcal M(V,\sqrt{\frac{2n-1}{\delta}},2n).$$ To prove that $ \mathcal S_{2k+1}=\mathcal S_{2k}$, we use the proof of Theorem 9.1 in \cite{CCT}. For the function $h^+$ constructed before Lemma \ref{harmonic-estimate}, we say that $\nabla h^+$ is an almost splitting direction. By Lemma \ref{almost-gradient-vector} below, we know that if $\nabla h^+$ is an almost splitting direction, $J\nabla h^+$ is also an almost splitting direction. So the the splitting direction is almost $J-$invariant. It follows that $ \mathcal S_{2k+1}=\mathcal S_{2k}$.
\end{proof}
\begin{lem} \label{almost-gradient-vector}
Under the conditions of Lemma \ref{harmonic-estimate-annual-1}, for a vector field $X$ on $A_p(a,b)$  which satisfies
\begin{align}\label{conditions}
|X|_{C^0(A_p(a,b))}\leq  C, \frac{1}{{\rm vol }(A_p(a,b))}\int_{A_p(a,b)}|\nabla X|^2 d\text{v}\leq \delta,
\end{align}
there exists a harmonic function $\theta$ defined in $A_p(a',b')$ such that
\begin{align}\label{gradient-est-app}
\frac{1}{{\rm vol }(A_p(a',b'))}\int_{A_p(a',b')}|\nabla\theta-X|^2 d\text{v}
< \Psi(\omega, \delta;a,b,a',b'),\end{align}
and
\begin{align}&\frac{1}{{\rm vol }A_p(a_1,b_1)}\int_{A_p(a_1,b_1)}|{\rm Hess }\,\theta|^2 d{\rm v}\notag\\
&\leq\Psi(\omega, \delta;a,b,a',b',a_1,b_1)\label{hessian-est},
\end{align}
where $A_p(a_1,b_1)$ is an even smaller annulus in $A_p(a',b')$.
\end{lem}

\begin{proof}
Let   $h$ be the harmonic function  constructed  in (\ref{f-harmonic-radial}) and $\theta_1=\langle X,\nabla h\rangle$.  Then
\begin{align}
\nonumber \nabla \theta_1=\langle\nabla X, \nabla h\rangle+\langle X,{\rm Hess }\,h\rangle,
\end{align}
It follows
\begin{align}
\nonumber &\int_{A_p(a,b)\cap \mathcal R}|\nabla\theta_1-X|^2d\text{v}\\
&\leq 2\int_{A_p(a,b)\cap \mathcal R}(\langle\nabla X, \nabla h\rangle^2d{\rm v}+\langle X,{\rm Hess }\,h-g\rangle^2)d{\rm v}.\notag
\end{align}
Thus by  (\ref{conditions}) and Lemma \ref{harmonic-estimate-annual-2}, we  get
\begin{align}\label{gradient-est}
\frac{1}{{\rm vol }(A_p(a,b))}\int_{A_p(a,b)\cap \mathcal R}|\nabla\theta_1-X|^2 d{\rm v}
\leq \Psi.\end{align}

Let $\theta$   be  a solution  of equation,
\begin{align}
\Delta\theta=0,~{ \rm in}~A_p(a,b)\cap \mathcal R, \end{align}
  with $\theta=\theta_1$   ~on  $\partial A_p(a,b)\cap \mathcal R$.
Then
\begin{align}
0&=\int_{A_p(a,b)}{\rm div } \left(\gamma_\eta(\theta-\theta_1)X\right)d{\rm v}\notag \\
&=\int_{A_p(a,b)}\left((\theta-\theta_1)\langle \nabla\gamma_\eta,X\rangle +\langle \nabla\theta-\nabla\theta_1,X\rangle+(\theta-\theta_1){\rm div } X\right)d{\rm v}\notag
\end{align}
Since $$\int_{A_p(a,b)}(\theta-\theta_1)\langle \nabla \gamma_\eta,X\rangle d{\rm v}\leq C\left(\int_{A_p(a,b)}|\nabla \gamma_\eta|^2{\rm div } \right)^{\frac{1}{2}}\leq C\sqrt{\eta},$$ taking $\eta\rightarrow 0$ we have
\begin{align}\label{divergence}
\int_{A_p(a,b)\cap \mathcal R}\langle \nabla\theta-\nabla\theta_1,X\rangle d{\rm v}\leq \Psi.
\end{align}
On the other hand, from
\begin{align}
0&=\int_{A_p(a,b)}{\rm div } \left( \gamma_\eta^2(\theta_1-\theta)\nabla \theta\right)\\
&=\int_{A_p(a,b)}\gamma_\eta^2\langle \nabla \theta_1-\nabla \theta,\nabla \theta\rangle+2\int_{A_p(a,b)}\gamma_\eta(\theta_1-\theta)\langle \gamma_\eta,\nabla\theta\rangle\notag\\
&\leq \int_{A_p(a,b)}\gamma_\eta^2\langle \nabla \theta_1-\nabla \theta,\nabla \theta\rangle+\sqrt{\eta}\int_{A_p(a,b)}\gamma_\eta^2|\nabla \theta|^2 d{\rm v}+\frac{C}{\sqrt{\eta}}\int_{A_p(a,b)}|\nabla \gamma_\eta|^2d{\rm v}, \notag
\end{align}
 we get
 \begin{align}
 (1-\sqrt{\eta})\int_{A_p(a,b)} \gamma_\eta^2|\nabla \theta|^2 d{\rm v}&\leq\int_{A_p(a,b)}  \gamma_\eta^2\langle \nabla\theta,\nabla\theta_1\rangle d\text{v}+C\sqrt{\eta}\notag\\
 &\leq \frac{1}{2}\int_{A_p(a,b)} \gamma_\eta^2|\nabla \theta|^2 d{\rm v}+\frac{1}{2}\int_{A_p(a,b)} \gamma_\eta^2|\nabla \theta_1|^2 d{\rm v}+C\sqrt{\eta}\notag.
\end{align}
Taking $\eta\rightarrow 0$, we have
$$\int_{A_p(a,b)} |\nabla \theta|^2 d{\rm v}\le \int_{A_p(a_2,b_2)} |\nabla \theta_1|^2 d{\rm v}<C.$$
Hence,
\begin{align}
\nonumber & \int_{A_p(a,b)}|\nabla \theta-X|^2d{\rm v}\\
&=\int_{A_p(a,b)}(|\nabla \theta|^2+|X|^2-2\langle \nabla\theta ,X\rangle)d{\rm v}\notag\\
\nonumber &=\int_{A_p(a,b)}(\langle\nabla \theta, \nabla\theta_1\rangle+|X|^2-2\langle \nabla\theta ,X\rangle)d{\rm v}&\\
&=\int_{A_p(a,b)}(\langle \nabla\theta_1-X,\nabla\theta\rangle+\langle X,X-\nabla\theta_1\rangle+\langle X,\nabla\theta_1-\nabla \theta\rangle)d{\rm v}.\notag
\end{align}
 Therefore, combining (\ref{conditions}) and (\ref{divergence}),
 we derive (\ref{gradient-est-app}) immediately.

To get (\ref{hessian-est}), we choose a cut-off function  $\phi$ which is supported in $A_p(a,b)$ with bounded gradient and
Lapalacian as in Lemma  \ref{cut-off}.   Then  by
 the Bochner identity,  we have
\begin{align}
 \int_{A_p(a,b)}\frac{1}{2}\phi\Delta|\nabla \theta|^2d{\rm v}=\int_{A_p(a,b)}\phi(|{\rm hess }\,\theta|^2+{\rm Ric }(\nabla \theta,\nabla \theta)) d{\rm v}.\notag
\end{align}
Since
\begin{align}
 \int_{A_p(a,b)}\frac{1}{2}\phi\Delta|X|^2d{\rm v}
 =-\int_{A_p(a,b)}\langle\nabla \phi,\langle X,\nabla X\rangle\rangle d{\rm v},\notag
\end{align}
 we obtain
\begin{align}
\nonumber \int_{A_p(a,b)}\phi(|{\rm Hess }\,\theta|^2  d{\rm v}& \leq \int_{A_p(a,b)}\frac{1}{2}\phi\Delta(|\nabla \theta|^2-|X|^2)d{\rm v}+C\delta.\notag&
\end{align}
Therefore,  by Lemma \ref{part1}, we  derive  (\ref{hessian-est}) from  (\ref{gradient-est-app}).
\end{proof}

Let $(M_i,\omega_i)$ be a sequence of conic K\"ahler-Einstein metrics in $\mathcal M(n,k,\delta,V)$:
$${\rm Ric}(\omega_i)\,=\,t_i\omega_i+\sum_{j=1}^k (1-\beta_i^j)D_i.$$
Assuming that there is $0<\epsilon, 0<T<1$ such that
$\beta_i^j\in [\epsilon, T]$, we can characterize the cone angle in the limit space.
\begin{prop}
If there is a tangent cone $T_x X\cong \mathbb{C}_{\bar\beta} \times \mathbb{R}^{2n-2}$, we have
$$1-\bar\beta=\sum_{j=1}^k m_i(1-\beta_\infty^j),$$
where $\beta_\infty^j$ is the limit of $\beta_i^j$.
\end{prop}
\begin{proof}
We use the arguments in \cite{tian15}. Assume that
$$(C_x,x,\omega_x)\,=\,\lim_{i\rightarrow \infty}(X,\frac{d}{r_i^2},p_i).$$
Using the estimate in section 3 and section 4, as Theorem 2.37 in \cite{CCT}, there are $\epsilon_i\rightarrow 0$ and maps $(\Phi_i,u_i): B_{p_i}(\frac{5}{2}, r_i^{-2}\omega_i)\rightarrow B_0(\frac{5}{2})$ such that
$$\int_{|z|\leq 1}|V(z)-2\pi \gamma|dz\leq \epsilon_i,$$
where $V(z)$ is the volume of $\Sigma_z=\Phi^{-1}(z)\cap u_i^{-1}[0,1]$. $K_M^{-1}$ restricts to a line bundle on $\Sigma_z$ whose curvature $\Omega$ is
$${\rm Ric}(r_j^{-2}\omega_j)\,=\,t_j\omega_j+2\pi\sum (1-\beta^i_j)D_i.$$
Let $\pi: S\Sigma_z\rightarrow \Sigma_z$ be the unit circle bundle, then
$$\pi^*\Omega=d\theta.$$ Since $K_M^{-1}$ is topologically equivalent to $T_{\Sigma_z}$, there is a section $v$ of $K_M^{-1}$ which is equal to the outward unit normal of $\partial \Sigma_z$ along the boundary of $\Sigma_z$ and has nondegenerate zeroes outside $\Sigma_z\setminus \cup D_i^j$. Put
$$s=\frac{v}{||v||}: \Sigma_z\setminus (D_i^j\cup v^{-1}(0))\rightarrow S\Sigma_z.$$
By Stokes theorem, we have
$$\int_{\Sigma_z\setminus \cup D_i^j}\Omega=\int_{\partial\Sigma_z}s^*\theta-\sum_{\nu(p)=0\text{ or }p\in \cup D_i^j}\lim_{\delta\rightarrow 0}\int_{\partial B_\delta(p,r_j^{-2}\omega_j)}s^*\theta.$$
It follows that $$\bar\beta-\chi(\Sigma_z)+\sum m_i(1-\beta_j^i)=o(1).$$
Since $\chi(\Sigma_z)\leq 1$, we must have $\chi(\Sigma_z)=1$, and
$$1-\bar\beta=\sum m_i(1-\beta_j^i)+o(1).$$
Taking $j\rightarrow \infty$,
we get
$$1-\bar\beta=\sum m_i(1-\beta_\infty^i).$$
\end{proof}

\end{document}